\documentclass[11pt]{amsart}
\usepackage{amsfonts}
\usepackage{amsmath}
\usepackage{amsthm}
\usepackage{url}
\usepackage[all]{xy}
\usepackage{latexsym}
\usepackage{amssymb}
\usepackage[cp850]{inputenc}
\usepackage{amsfonts}
\usepackage{graphicx}
\usepackage{dsfont}
\usepackage{float}
\usepackage{labelfig}

\usepackage{color}

  \definecolor{verydarkblue}{rgb}{0,0,0.5}
  \usepackage[
      breaklinks,
      colorlinks,
      citecolor=verydarkblue,
      linkcolor=verydarkblue,
      urlcolor=verydarkblue,
      hyperindex
  ]{hyperref}

\newtheorem{sat}{Theorem}[section]		
\newtheorem{lem}[sat]{Lemma}
\newtheorem{kor}[sat]{Corollary}			
\newtheorem{prop}[sat]{Proposition}
\newtheorem{defi}{Definition}
\newtheorem*{defi*}{Definition}			
\newtheorem*{bei*}{Example}
\newtheorem*{sat*}{Theorem}				
\newtheorem*{kor*}{Corollary}
\newtheorem*{rmk*}{Remark}				
	
\newtheorem*{quest*}{Question}	
	
\newtheorem*{claim}{Claim}	

\let\ssection=\section
\renewcommand{\section}{\setcounter{equation}{0}\ssection}

\newtheorem*{namedtheorem}{\theoremname}
\newcommand{\theoremname}{testing}
\newenvironment{named}[1]{\renewcommand{\theoremname}{#1}\begin{namedtheorem}}{\end{namedtheorem}}

\theoremstyle{remark}
\newtheorem*{bem}{Remark}

\newtheorem*{namedtheoremr}{\theoremnamer}
\newcommand{\theoremnamer}{testing}

\newcommand{\BR}{\mathbb R}			
\newcommand{\BN}{\mathbb N}			
			\newcommand{\BZ}{\mathbb Z}

\newcommand{\BI}{\mathbb I}

\newcommand{\CA}{\mathcal A}

\newcommand{\CG}{\mathcal G}		
		
		\newcommand{\CL}{\mathcal L}
\newcommand{\CM}{\mathcal M}		
		\newcommand{\CP}{\mathcal P}
		
\newcommand{\CS}{\mathcal S}		\newcommand{\CT}{\mathcal T}

\newcommand{\actson}{\curvearrowright}
\newcommand{\D}{\partial}

\DeclareMathOperator{\Map}{Map}

\newcommand{\comment}[1]{}

\DeclareMathOperator{\Homeo}{Homeo}

\DeclareMathOperator{\const}{const}

\DeclareMathOperator{\supp}{supp}

\newcommand{\fsubd}{\mathrel{{\scriptstyle\searrow}\kern-1ex^d\kern0.5ex}}
\newcommand{\bsubd}{\mathrel{{\scriptstyle\swarrow}\kern-1.6ex^d\kern0.8ex}}

\renewcommand{\epsilon}{\varepsilon}
\renewcommand{\le}{\leqslant}
\renewcommand{\ge}{\geqslant}
\renewcommand{\emptyset}{\varnothing}

\begin{document}

\title[]{Mapping class group orbit closures for non-orientable surfaces}
  \author[V. Erlandsson]{Viveka Erlandsson}
  \address{School of Mathematics, University of Bristol \\ Bristol BS8 1UG, UK {\rm and}  \newline ${ }$ \hspace{0.2cm} Department of Mathematics and Statistics, UiT The Arctic University of  \newline ${ }$ \hspace{0.2cm} Norway}
  \email{v.erlandsson@bristol.ac.uk}
  \thanks{The first and third authors gratefully acknowledge support from EPSRC grant EP/T015926/1.}
    \author[M. Gendulphe]{Matthieu Gendulphe}
    \address{}
  \email{matthieu@gendulphe.com}
  \author[I. Pasquinelli]{Irene Pasquinelli}
  \address{School of Mathematics, University of Bristol \\ Bristol BS8 1UG, UK}
  \email{irene.pasquinelli@bristol.ac.uk}
\author[J. Souto]{Juan Souto}
\address{UNIV RENNES, CNRS, IRMAR - UMR 6625, F-35000 RENNES, FRANCE}
\email{jsoutoc@gmail.com}

\begin{abstract}
Let $S$ be a connected non-orientable surface with negative Euler characteristic and of finite type. We describe the possible closures in $\CM\CL$ and $\CP\CM\CL$ of the mapping class group orbits of measured laminations, projective measured laminations and points in Teichm\"uller space. In particular we obtain a characterization of the closure in $\CM\CL$ of the set of weighted two-sided curves. 
\end{abstract}

\maketitle

\section{Introduction}

In this paper we study the closures of the orbits of the action of the mapping class group of a complete connected yperbolic surface on the space of measured laminations, projective measured laminations and Teichm\"uller space. When reading this the reader might well be surprised. They might be thinking that all of this is already known; that it is a classical results that the action of the mapping class group on $\CP\CM\CL$ is minimal in the sense that all orbits are dense \cite{FLP}, that $\CP\CM\CL$ is the limit set of the action of the mapping class group on Teichm\"uller space, and that the orbit closures of the action of $\Map$ on $\CM\CL$ were already described by Mirzakhani and Lindenstrauss \cite{Lindenstrauss-Mirzakhani}. And the reader would be correct if they restricted themselves to orientable surfaces. Things are actually quite different in the non-orientable world.

For starters, it is due to Scharlemann \cite{Scharlemann} and Danthony-Nogueira \cite{Danthony-Nogueira} that the set of (projective) measured laminations which have a one-sided closed leaf is open and has full measure, where a simple closed essential curve is {\em one-sided} if it has a regular neighborhood homeomorphic to a M\"obius band---otherwise it is {\em two-sided}. Since the set of measured laminations without one-sided component is also mapping class group invariant, we get in particular that the action $\Map(S)\actson\CP\CM\CL(S)$ is not minimal whenever $S$ is non-orientable: the set
$$\CP\CM\CL^+(S)=\{\lambda\in\CP\CM\CL(S)\text{ without closed one-sided components}\}$$
is a non-empty, closed, invariant proper subset. Our first result is that the action of the mapping class group on $\CP\CM\CL^+(S)$ is minimal, or rather that this set is the unique closed minimal subset of $\CP\CM\CL(S)$:

\begin{sat}\label{thm minimal set}
  Let $S$ be a connected, possibly non-orientable, non-exceptional hyperbolic surface of finite topological type. The set $\CP\CM\CL^+(S)$ is the unique non-empty closed subset of $\CP\CM\CL(S)$ which is invariant and minimal under the action of $\Map(S)$.
\end{sat}

In Theorem \ref{thm minimal set}, as in the remaining of the paper, we say that a hyperbolic surface $S$ is {\em exceptional} if it is either a pair of pants or non-orientable with $\chi(S) = -1$. Otherwise the surface is {\em non-exceptional}. The non-orientable exceptional surfaces 
are the two-holed projective plane, the one-holed Klein bottle, and the connected sum of three projective planes. The mapping class group, the spaces of measured laminations and projective measured laminations are well understood if the surface is exceptional, and the interested reader will have no difficulty to clarify matters on their own for those cases. In any case we discuss briefly exceptional surfaces in section \ref{subsec exceptional} below.
\medskip

The closure of the mapping class group orbit of a two-sided curve $\gamma$ is a closed invariant subset of $\CP\CM\CL^+$. Theorem \ref{thm minimal set} implies thus that $\CP\CM\CL^+(S)=\overline{\Map(S)\cdot\gamma}$. It follows a fortiori that the set of projective classes of two-sided curves is dense in $\CP\CM\CL^+(S)$. This answers a question which seems to have been around for some time \cite{Bestvina}:

\begin{sat}\label{main theorem}
  Let $S$ be a connected, possibly non-orientable, non-exceptional hyperbolic surface of finite topological type. The set of two-sided curves is dense in $\CP\CM\CL^+(S)$.
\end{sat}

If Theorem \ref{main theorem} follows from Theorem \ref{thm minimal set} because $\CP\CM\CL^+$ is minimal, the fact that it is actually the unique non-empty closed minimal subset of $\CP\CM\CL(S)$ implies that $\CP\CM\CL^+(S)$ is contained in the closure of every orbit $\lambda\in\CP\CM\CL(S)$. To describe the actual closure of $\Map(S)\cdot\lambda$ for an arbitrary $\lambda$ we need to introduce some notation.

Following Lindenstrauss-Mirzakhani \cite{Lindenstrauss-Mirzakhani} we consider the decomposition $\lambda=\gamma_\lambda+\lambda'$ where $\gamma_\lambda$ is the atomic part of $\lambda$, and let $R_\lambda$ be the union of those connected components of $S\setminus\gamma_\lambda$ which contain a non-compact leaf of $\lambda$. The pair $(R_\lambda,\gamma_\lambda)$ is, in the terminology of \cite{Lindenstrauss-Mirzakhani}, a {\em complete pair}. We associate to the pair $(R_\lambda,\gamma_\lambda)$ first the set $\gamma_\lambda+\CM\CL^+(R_\lambda)$ of measured laminations of the form $\gamma_\lambda+\mu$ where $\mu\in\CM\CL^+(R_\lambda)$ is a measured lamination supported by $R_\lambda$ and without one-sided leafs, and then the orbit of this set under the mapping class group action:
$$\CG_\lambda=\cup_{\phi\in\Map(S)}\phi\big(\gamma_\lambda+\CM\CL^+(R_\lambda)\big)\subset\CM\CL(S).$$
Unsurprisingly we denote by $\CP\CG_\lambda\subset\CP\CM\CL(S)$ the image of $\CG_\lambda$ in the space of projective measured laminations.

With all this notation in place we can give a precise description for the closure of $\Map(S)\cdot\lambda$ in $\CP\CM\CL$ for an arbitrary $\lambda$.

\begin{sat}\label{thm orbit closures PML}
    Let $S$ be a connected, possibly non-orientable, non-exceptional hyperbolic surface of finite topological type. We have $\overline{\Map(S)\cdot\lambda}=\CP\CG_\lambda\cup\CP\CM\CL^+(S)$ for any projective measured lamination $\lambda\in\CP\CM\CL(S)$.
\end{sat}

The reader might be wondering what happens if we consider the action of the mapping class group on $\CM\CL(S)$ instead. In fact, even if this question had not crossed their mind it would have to be considered: understanding orbit closures for the action $\Map(S)\actson\CM\CL(S)$ in an integral part of the proof of Theorem \ref{thm orbit closures PML}. We prove:

\begin{sat}\label{thm orbit closures ML}
  Let $S$ be a connected, possibly non-orientable, non-exceptional hyperbolic surface of finite topological type. We have $\overline{\Map(S)\cdot\lambda}=\CG_\lambda$ for any measured lamination $\lambda\in\CM\CL(S)$.
\end{sat}

In the orientable case, Theorem \ref{thm orbit closures ML} is due to Lindenstrauss-Mirzakhani (see Theorem 1.2 in \cite{Lindenstrauss-Mirzakhani}). They obtain it as a consequence of their main theorem, the classification of mapping class group invariant measures on $\CM\CL(S)$. It seems that, for the time being, the latter kind of result is out of scope in the non-orientable setting---for example, everything that Lindenstrauss and Mirzakhani do, relies on the fact that the moduli space of $S$ has finite volume, and this is false if the surface is non-orientable. The proof of Theorem \ref{thm orbit closures ML} that we give below is actually pretty elementary and can also be enjoyed by those readers who only care about orientable surfaces.
\medskip

As we just pointed out, the moduli space of a non-orientable surface has infinite volume. Indeed, the usual analogy between mapping class group and an arithmetic lattice breaks in the non-orientable case. In fact, in the non-orientable case it might well be closer to the mark to compare the mapping class group with a Kleinian group of the second kind, or rather with some higher rank version of this. Recall that a Kleinian group is of the second kind if its discontinuity domain is not empty, or equivalently if its limit set is a proper subset of the boundary of hyperbolic space. Thinking of the Thurston boundary $\D\CT(S)=\CP\CM\CL(S)$ of Teichm\"uller space as being the analogue to the boundary at infinity of hyperbolic space, we interpret the following result as describing the limit set of the action of the mapping class group on Teichm\"uller space:

\begin{sat}\label{thm limit set}
  Let $S$ be a connected, possibly non-orientable, non-exceptional hyperbolic surface of finite topological type. We have $\overline{\Map(S)\cdot X}\cap\D\CT(S)=\CP\CM\CL^+(S)$ for any point $X$ in Teichm\"uller space $\CT(S)$.
\end{sat}

After seeing all these beautiful theorems announced, the reader is surely burning in desire to read the proofs. Let us release some of the preassure in their chest and explain how to prove that filling and uniquely ergodic measured laminations are limits of two-sided curves. 

\subsection*{An argument}
Let us suppose that $\lambda$ is a filling and uniquely ergodic measured lamination. As every projective measured lamination, $\lambda$ can be approximated by a sequence $(\alpha_i)$ of weighted multicurves. Some of these might well be one-sided but, invoking for example the classification theorem of surfaces, we can get for all $i$ a two-sided simple curve $\beta_i$ with $\iota(\beta_i,\alpha_i)=0$. Passing to a subsequence we can assume that the sequence $(\beta_i)$ converges projectively to some $\mu\in\CP\CM\CL(S)$. Continuity of the intersection number implies then that $\iota(\mu,\lambda)=0$. Now, since $\lambda$ is filling we get that both $\mu$ and $\lambda$ have the same support. Since it is also uniquely ergodic we get that $\lambda$ and $\mu$ differ by a factor, meaning that they are equal in $\CP\CM\CL$. This proves that  the filling uniquely ergodic lamination $\lambda\in\CP\CM\CL(S)$ is the limit of the sequence of two-sided curves $(\beta_i)$. 
\medskip

At this point the reader might be ready to point out we also proved that every element in the closure of the set of filling uniquely ergodic measured laminations is also a limit of two sided curves, thinking that the set of filling uniquely ergodic measured laminations must be dense. Well, they are dense in say $\CP\CM\CL$ if the surface is orientable, but as we pointed out earlier this is not true for non-orientable surfaces. In the non-orientable case we get from say Theorem \ref{thm minimal set} that the set of filling uniquely ergodic laminations is dense in $\CP\CM\CL^+$. If the reader knew how to prove this without using the results of this paper, then everything we do here could be done in 5 to 10 pages. Indeed, the bulk of the work is to deal with measured laminations which are not ergodic.

\subsection*{Section-by-section summary}

Once this introduction is concluded, we have a section on preliminaries. We recall a few well-known facts about topology of surfaces, mapping class groups, laminations, and hyperbolic surfaces. There is nothing new here but the reader who is not used to thinking about non-orientable surfaces might still want to have a look.

In Section \ref{sec:uniform train tracks} we recall what train tracks are and introduce what we call {\em uniform}, or rather {\em $(C,\lambda)$-uniform} train tracks. In a nutshell, these are train tracks carrying a given lamination $\lambda$ and where, from the point of view of $\lambda$, all edges have comparable lengths. Existence of such train tracks, or rather the fact that every train track can be refined to such a train track, will be proved in Appendix \ref{appendix}. What we do in Section \ref{sec:uniform train tracks} is to prove two technical results---needed later on---about such uniform train tracks. We think that such uniform train tracks might turn out to be useful in other settings as well.

In any case, uniform train tracks play a key role in the proof of the following theorem in Section \ref{sec:coloring}:

\begin{named}{Theorem \ref{prop lm}}[Informal version]
  Let $\lambda\in\CL(S)$ be a lamination and let $\mu_1,\mu_2$ be distinct unit length ergodic measured laminations with support $\lambda$. If $\tau$ is a sufficiently nice uniform train track carrying $\lambda$ then there are disjoint non-empty sub-train tracks $\tau_1$ and $\tau_2$ such that any unit length measured lamination carried by $\tau_i$ is near $\mu_i$.
\end{named}

The actual statement of Theorem \ref{prop lm} is pretty technical---the informal version given here just captures the gist of it.

\begin{bem}
Theorem \ref{prop lm} is motivated by a result of Lenzhen-Masur \cite[Proposition 1]{Lenzhen-Masur}. It is not completely clear to us if, as stated, we could just have used what they prove instead of working with Theorem \ref{prop lm}. But in any case we would have had to comment on the proof because they work in the world of quadratic holomorphic differentials and hence, at least formally, deal only with orientable surfaces. We found however much simpler to just give a direct proof of Theorem \ref{prop lm}. Maybe a reader who is better versed than ourselves in taking geometric limits and exploiting compactness properties of spaces of manifolds would be able to see clearly enough through their argument to modify the proof so that it also works in the non-orientable case.
\end{bem}

Theorem \ref{prop lm} is maybe the main technical result of this paper and will play a key role in the proof of Theorem \ref{main theorem}. It has however applications that might be of independent interest. For example we obtain the following as a corollary:

\begin{named}{Corollary \ref{kor number of measures}}
Let $c(S)$ and $c^+(S)$ be, respectively, the maximal number of components of a multicurve and of a two-sided multicurve in $S$. Every lamination $\lambda\in\CL(S)$ supports at most $c(S)$ mutually singular ergodic transverse measures. Moreover, if $\lambda$ has no one-sided leaves then it supports at most $c^+(S)$ mutually singular ergodic transverse measures.
\end{named}

If $S$ has non-orientable genus $k$ and $r$ boundary components then we have 
$$c(S)=2k-3+r\text{ and }c^+(S)=\left\{
\begin{array}{ll}
\frac 12(3k-7+2r) & \text{ if }k\text{ is odd}\\
\frac 12(3k-8+2r) & \text{ if }k\text{ is even}
\end{array}
 \right\}, $$
meaning that $c(S)$ is about $k/2$ larger than $c^+(S)$.

\begin{bem}
Corollary \ref{kor number of measures} is due to Levitt \cite{Levitt} in the orientable case. He uses a beautiful idea due to Katok \cite{Katok} and we really encourage the reader to have a look at Levitt's argument---it is very nice mathematics.
\end{bem}

Let us continue now with the summary of the paper. In Section \ref{sec:two-sided} we prove Theorem \ref{main theorem}. Indeed, although we might have given the impression that Theorem \ref{main theorem} was a corollary of Theorem \ref{thm minimal set}, it is rather the other way around: all the theorems mentioned earlier build on Theorem \ref{main theorem}. The main idea of its proof is as follows: Let $\mu\in\CM\CL$ be a representative of the element in $\CP\CM\CL^+$ that we want to approximate and consider its ergodic decomposition. For the sake of concreteness say that $\mu=\mu_1+\mu_2$ with $\mu_i$ ergodic. Let then $\tau$ be a ``sufficiently nice'' train track carrying the support of $\mu$ and let $\tau_1,\tau_2\subset\tau$ be the sub-train tracks provided by Theorem \ref{prop lm}. We can then approximate $\mu$ by weighted (simple) multicurves $c_1\cdot\gamma_1+c_2\cdot\gamma_2$ with $\gamma_i$ carried by $\tau_i$. But we must make sure that we can choose the $\gamma_i$ to be two-sided. To do that we will need to first characterize which train tracks do not carry any two-sided curves and use then a (rather minimal) quantification of Scharlemann's result on the openness of the set of measured lamination which have a one-sided component.

Armed with Theorem \ref{main theorem} we attack the other results stated in the introduction. In Section \ref{sec:orbit closures ML} we prove Theorem \ref{thm orbit closures ML}. The proof only relies on the density of two-sided curves, that is Theorem \ref{main theorem}, and of elementary facts about measured laminations. In Section \ref{sec:orbit closures PML} we deduce  Theorem \ref{thm orbit closures PML} and Theorem \ref{thm minimal set} from Theorem \ref{thm orbit closures ML}. Finally, in Section \ref{sec:orbit closures Tecih} we prove Theorem \ref{thm limit set}.

As we mentioned earlier, we conclude with an appendix in which we prove the existence of uniform train tracks.

\subsection*{On the genesis of this paper}
This paper grew out of the work of the second of us. In fact, one of the inclusions in the orbit closure results mentioned above were proved by the second author in the unpublished preprint \cite{Matthieu}, and equality was established for surfaces of genus one. The genus one assumption guarantees that the complement of a one-sided curve is orientable and hence that one can use all the standard theory to deal with measured laminations disjoint from such a curve. The novel ingredients here are Theorem \ref{main theorem} and the fact that we give a proof of Theorem \ref{thm orbit closures ML} which does not rely on the Lindenstrauss-Mirzakhani classification of mapping class group invariant measures on the space of measured laminations.

It should be said that the orbits closure theorems were only a fraction of the content of \cite{Matthieu}---we hope to revisit other parts of the said paper at a later point.

\subsection*{Another paper}
At the moment of posting this paper to the arXiv we noticed that just a couple of days earlier Khan \cite{Khan} had posted a very nice paper on this topic. Some of his results are of compeltely different nature than what we are doing here: for example, in Theorem 5.2 he shows that certain subsets of Teichm\"uller space are non quasi-convex, answering in the negative a question asked in \cite{Matthieu}. Other results are however somewhat weaker results than the ones we present here. For example, his Theorem 3.3, although stated with different words, asserts that certain kinds of measured laminations, for example orientable ergodic laminations, are limits of two-sided curves. It goes without saying that prior to the completion of our respective papers, neither him nor ourselves were aware of each other's work.

We really encourage the reader to have a look at Khan's paper because the arguments he gives when dealing with particular cases of laminations are much prettier and simpler than what awaits them if they continue reading, following their way to the city of woe, to eternal pain, tra perduta gente.

\subsection*{Acknowledgements}
The last author would like to thank Anna Lenzhen for interesting conversation and pointing him to the paper \cite{Lenzhen-Masur}---it was an incredibly useful hint.

\section{Preliminaries}\label{sec:preli}
We assume that the reader feels at home in the world of surfaces, that they know what the mapping class group is, that they are familiar with laminations, measured laminations and Teichm\"uller spaces. We breeze through these concepts below, but should the reader wish more details we refer to the beautiful books \cite{Farb-Margalit, Casson-Bleiler, Penner-Harer}. However, even if the reader knows all these things inside out, it might well be that they have only encountered them in the orientable setting. This is why we start by recalling some basic facts about the topology of non-orientable surfaces.

\subsection{Non-orientable surfaces}
We should probably say once what a surface is: it is nothing other than a 2-dimensional smooth manifold, possibly with non-empty boundary. Moreover, unless we say explicitly otherwise, we will assume that our surfaces $S$ are connected and have {\em finite topological type}, by what we understand that $S$ is obtained from a compact surface $\Sigma$ by deleting some of its boundary components, possibly none and possibly all. We will indeed be mostly interested in the case that the surface $S$ has empty boundary, meaning that $S=\Sigma\setminus\D\Sigma$ where $\Sigma$ is a compact surface.

As in the orientable case, the classification theorem is the foundation on which one builds any work on surfaces:

\begin{named}{Classification theorem}
  Two compact connected surfaces $\Sigma$ and $\Sigma'$ are homeomorphic if and only if the following two hold:
  \begin{itemize}
  \item $\Sigma$ and $\Sigma'$ have the same number of boundary components and the same Euler characteristic.
  \item $\Sigma$ and $\Sigma'$ are either both orientable or both non-orientable.
  \end{itemize}
\end{named}
  
  We refer for example to \cite{Stillwell} for a proof of this theorem (under the, always satisfied, assumption that the involved surfaces are triangulable).

%

A simple curve in a surface is {\em essential} if it neither bounds a disk, nor an annulus, nor a M\"obius band. Equivalently, a simple curve is essential if and only if it is homotopically non-trivial, non-peripheral, and primitive. Note that, as long as $S$ has negative Euler characteristic, every orientation reversing simple curve is essential. Note also that the boundary of a regular neighborhood of such an orientation reversing simple curve, although it is simple and homotopically non-trivial, is not essential. In the sequel we will say that an essential simple curve is {\em one-sided} if it reverses orientation and {\em two-sided} if its preserves orientation. A {\em topological multicurve} is the union of pairwise non-isotopic disjoint simple curves---it is essential, (resp. two-sided, resp. one-sided) if all its components are.

  To conclude with these topological preliminaries, recall that we say that a surface $S$ with negative Euler characteristic is {\em exceptional} if it is either a pair of pants or non-orientable with $\chi(S) = -1$. For clarification, here is a list of the possible exceptional surfaces:
  \begin{itemize}
  \item {\em $S$ orientable:} Pair or pants
  \item {\em $S$ non-orientable:} Two-holed projective plane, one-holed Klein bottle, connected sum of three projective planes.
  \end{itemize}
The main reason to ignore exceptional surfaces is that, in all cases except the connected sum of three projective planes, they support no measured laminations---at least none containing non-compact leaves. Also the case of connected sum of three projective planes is a bit special, and for technical reasons we exclude it since sometimes we need to have space enough to find curves with certain properties. For example we will want to know the following:

\begin{lem}\label{lem filling pair}
  If $S$ is a non-exceptional connected surface then there are two-sided simple curves $\alpha,\beta\subset S$ which fill $S$.
  \end{lem}

Recall that a collection $\{\alpha_1,\dots,\alpha_r\}$ of curves {\em fills} $S$ if $\sum_i\iota(\alpha_i,\gamma)>0$ for every essential curve $\gamma$---here $\iota(\cdot,\cdot)$ is the geometric intersection number.
\medskip

We are sure that Lemma \ref{lem filling pair} is well-known and that is why we just sketch the proof in the case of non-orientable surfaces. Via the classification theorem of surfaces we get that there is $\eta\subset S$ consisting of either a single one-sided simple curve (if $S$ has odd genus) or two disjoint one-sided simple curves (if the genus is even) such that $S\setminus\eta$ is connected and orientable. For the sake of concreteness we assume that we are in the former case, leaving the other case to the reader.

  Let $\Sigma$ be the metric completion of $S\setminus\eta$ and note that a connected component of $\D\Sigma$ maps in a two-to-one way onto $\eta$---the remaining components are mapped homeomorphically onto $\D S$. Choose two distinct points $x,y\in\eta$ and let $x',x''$ and $y',y''$ be the corresponding points in $\D\Sigma$. Since $\chi(\Sigma)=\chi(S\setminus\eta)=\chi(S)\le-2$ we can find two disjoint, non-boundary parallel simple arcs $\kappa',\kappa''\subset\Sigma$ with endpoints $\D\kappa'=\{x',y'\}$ and $\D\kappa''=\{x'',y''\}$. When mapped into $S$, the two arcs $\kappa'$ and $\kappa''$ match together and yield a two-sided curve $\alpha\subset S$.

  Now let $\beta'\subset S\setminus\eta$ be any essential curve which fills with $\kappa'\cup\kappa''$. The image $\beta$ of $\beta'$ in $S$ is the second curve we are after. This concludes the discussion of Lemma \ref{lem filling pair}.\qed
  \medskip

\subsection{Mapping class groups}
The mapping class group of a surface $S$ is the group
$$\Map(S)=\Homeo(S)/\Homeo_0(S)$$
of isotopy classes of homeomorphisms of $S$---the reader can replace homeomorphisms by diffeomorphisms if they so wish, and nothing will change.

Both in the orientable and in the non-orientable case, the prime example of an element in the mapping class group $\Map(S)$ is a {\em Dehn-twist} along a two-sided curve $\gamma$. A Dehn-twist $D_\gamma$ along $\gamma$ is a mapping class represented by a homeomorphism obtained as follows: conjugate the homeomorphism
$$(\BR/\BZ)\times[0,1]\to(\BR/\BZ)\times[0,1],\ \ (\theta,t)\mapsto(\theta+t,t)$$
 via a homeomorphism between the standard annulus $(\BR/\BZ)\times[0,1]$ and a regular neighborhood $A$ of $\gamma$, and extend the so obtained homeomorphism $A\to A$ by the identity to a self-homeomorphism $S\to S$.

If $\gamma$ and $\gamma'$ are disjoint two-sided curves then they have disjoint regular neighborhoods. This means that Dehn-twists along $\gamma$ and $\gamma'$ commute. 
 
 Below we will need to understand what happens to certain curves when we apply Dehn-twists to them. The following lemma due to Ivanov \cite[Lemma 4.2]{Ivanov} is priceless:

 \begin{lem}[Ivanov]\label{lem ivanov}
   Let $\alpha$ and $\beta$ be simple essential curves and assume that $\gamma=\gamma_1\cup\dots\cup\gamma_r$ is a two-sided simple topological multicurve. Let $D_{\gamma_i}$ be a Dehn-twist along $\gamma_i$ and for some $n_1,\dots,n_r\in\BZ$ let $T=D_{\gamma_1}^{n_1}\circ\ldots\circ D_{\gamma_r}^{n_r}$. Then we have
   \begin{align*}
   \iota(T(\alpha),\beta)&\le\left(\sum_i\vert n_i\vert\cdot\iota(\alpha,\gamma_i)\cdot\iota(\gamma_i,\beta)\right)+\iota(\alpha,\beta)\\
   \iota(T(\alpha),\beta)&\ge\left(\sum_i(\vert n_i\vert-2)\cdot\iota(\alpha,\gamma_i)\cdot\iota(\gamma_i,\beta)\right)-\iota(\alpha,\beta)
   \end{align*}
\end{lem}

 Since Ivanov only considers the orientable case, let us comment briefly on the proof in the non-orientable case. First, note that the upper bound for $\iota(T(\alpha),\beta)$ is just the number that you get by applying the standard representative of $T$ to $\alpha$ and counting intersection points without bothering to get rid of bigons. To get the lower bound lift everything to the orientation cover $S'\to S$ and apply Ivanov's lemma there.\qed
\medskip 

Note now that, since measured laminations and a fortiori multicurves are determined by intersection numbers, we get from Ivanov's lemma, and with the same notation as therein, that 
$$\lim_{k\to\infty}\frac 1k\cdot T^k(\alpha)=\sum_i \big(\vert n_i\vert\cdot\iota(\alpha,\gamma_i)\big)\cdot\gamma_i.$$
Since for every multicurve $\gamma=\gamma_1\cup\dots\cup\gamma_r$ we can find a two-sided curve $\alpha$ with $\iota(\alpha,\gamma_i)\neq 0$ for all $i$ we get then the following well-known fact:

\begin{lem}\label{lem approximate multicurves by curves}
Every two-sided weighted multicurve is a limit of weighted two-sided curves.\qed
\end{lem}

Dehn-twists fix many curves: the curve one is twisting around, as well as all the curves disjoint from it. Pseudo-Anosov mapping classes are on the other extremum. Recall that, following Thurston, a mapping class is {\em pseudo-Anosov} if none of its positive powers fixes the homotopy class of an essential curve. As in the orientable setting, examples of pseudo-Anosov mapping classes can be constructed as follows: let $D_{\alpha}$ and $D_{\beta}$ be Dehn-twists along two two-sided curves $\alpha$ and $\beta$ which together fill the surface $S$ and consider the composition $D_\alpha^n\circ D_\beta^n$. It follows (compare with \cite{Johanna}) for example from Lemma \ref{lem ivanov} that $D_\alpha^n\circ D_\beta^n$ is pseudo-Anosov for all large $n$. Alternatively, the reader can consider the lifts of $D_\alpha$ and $D_\beta$ to the orientation cover and quote the result in the orientated case. Either way we get:

\begin{lem}\label{lem existence PA}
If $S$ is a connected non-essential hyperbolic surface then $\Map(S)$ contains pseudo-Anosov elements.\qed
  \end{lem}

Non-orientable surfaces are confusing. This is why we add here some comments to put things into context, although none of these things will play a role in the sequel:
\smallskip

\noindent{\bf 1.)} Note that if $S$ is orientable then what we here called the mapping class group is ofter referred to as the {\em full mapping class group}, with the mapping class group itself being the subgroup consisting of elements represented by orientation preserving homeomorphisms. That last sentence makes no sense for non-orientable surfaces.

\noindent{\bf 2.)} The reader might have been disturbed because we wrote ``a'' instead of ``the'' in the sentence ``a Dehn-twist along $\gamma$''. It is not a typo. The point is that, even up to isotopy, there is no uniqueness for the homeomorphism between $(\BR/\BZ)\times[0,1]$ and the regular neighborhood $A$ of our given two-sided curve $\gamma$, the homeomorphism we used to define the Dehn-twist. In the orientable world one can decide to choose it to be orientation preserving. This leads to the notion of {\em right Dehn-twist}. If one chooses it to be orientation reversing then one has the {\em left Dehn-twist}---these notions make no sense in the non-orientable world and hence ``a'' is the correct article.

\noindent{\bf 3.)} Recall that the adjective ``two-sided'' only applies to essential curves. If we were to mimic the construction of a Dehn-twist but starting with a non-essential orientation preserving simple curve which is the boundary of a M\"obius band then we would get the trivial mapping class. On the other hand we get for example from Lemma \ref{lem ivanov} that Dehn-twists along two-sided curves have infinite order in $\Map(S)$. 

\noindent{\bf 4.)} There is nothing like a Dehn-twist along a one-sided curve. In fact, every homeomorphism of a surface which fixes the complement of a M\"obius band is isotopic to the identity.
\medskip

We refer to \cite{Farb-Margalit} for facts on the mapping class group, and to \cite{Paris-notes} and the references therein for facts on mapping class groups of specifically non-orientable surfaces.

\subsection{Hyperbolic metrics}

It follows from the classification theorem of surfaces, together with elementary constructions in hyperbolic geometry, that every connected surface of negative Euler characteristic admits a complete hyperbolic metric with totally geodesic boundary and finite volume. In fact, the Gau\ss-Bonnet theorem implies that the condition on the Euler-characteristic is not only sufficient but also necessary for the existence of such a metric. Under a hyperbolic surface, orientable or not, we will understand in the sequel a surface endowed with such a complete hyperbolic metric with (possibly empty) totally geodesic boundary and finite volume.

As in the orientable setting, we let the Teichm\"uller space $\CT(S)$ of $S$ be the set of all isotopy classes of hyperbolic metrics on $S$. The pull-back action of the group of diffeomorphisms on the set of hyperbolic metrics induces an action $\Map(S)\actson\CT(S)$ of the mapping class group on Teichm\"uller space. 

Both in the orientable and in the non-orientable settings one can come up with many reasonable ways to endow $\CT(S)$ with a topology. With the proviso that what looks reasonable to somebody, might look bonkers to everybody else, all those definitions agree. The obtained topology on $\CT(S)$ is mapping class group invariant and $\CT(S)$ becomes homeomorphic to a euclidean space. If for example one endows $\CT(S)$ with the topology induced by the embedding via length functions into the space $\BR_+^{\CS(S)}$ where $\CS(S)$ is the set of essential curves, then a concrete homeomorphism between $\CT(S)$ and the euclidean space of appropriate dimension is obtained via Fenchel-Nielsen coordinates. See \cite{Papa-Penner} for a discussion on the Teichm\"uller space and on Fenchel-Nielsen coordinates for non-orientable surfaces.

\subsection{Laminations and measured laminations}
Let $S$ be a hyperbolic surface with empty boundary, orientable or not, and recall that we require it to be complete and of finite volume. A {\em lamination} is a compact subset of $S$ which admits a decomposition into disjoint simple geodesics. Those geodesics are unique---they are the {\em leaves} of the lamination. A basic example of a lamination on $S$ is the closed geodesic isotopic to an simple essential curve. Simple geodesic multicurves are a slight generalization of this example: a {\em simple geodesic multicurve} is a lamination obtained by taking the collection of simple closed geodesics corresponding to an essential simple topological multicurve.

The set $\CL(S)$ of all laminations on $S$ is compact when endowed with the topology induced by the Hausdorff distance. Also, if $\phi:S\to S'$ is a homeomorphism between hyperbolic surfaces then there is a homeomorphism $\phi_*:\CL(S)\to\CL(S')$ sending simple closed geodesics $\gamma$ to the geodesic freely homotopic to $\phi(\gamma)$. Moreover, homotopic homeomorphisms $S\to S'$ induce the same homeomorphism $\CL(S)\to\CL(S')$. It follows that there is an action $\Map(S)\actson\CL(S)$ of the mapping class group on the space $\CL(S)$ of laminations.
\medskip

A {\em measured lamination} $\lambda$ is a lamination $\lambda_{\CL}$ endowed with a transverse measure. Recall that a {\em transverse measure} is an assignment of a Radon measure $\lambda^I$ on every segment $I$ transversal to the lamination in such a way that segments which are isotopic relative to the lamination have the same measure
$$\iota(\lambda,I)\stackrel{\text{def}}=\lambda^I(I).$$
The {\em support} $\supp(\lambda)$ is the smallest sublamination of $\lambda_{\CL}$ such that $\lambda^I=0$ for every segment $I$ with $I\cap\supp(\lambda)=\emptyset$. The set of all measured laminations in $S$ is denoted by $\CM\CL(S)$.

Again, the most basic example of a measured lamination is given by an essential simple closed geodesics $\gamma$: the underlying lamination is the geodesic $\gamma$ itself and the transverse measure is such that it assigns to the segment $I$ the counting measure of the set $I\cap\gamma$. We will often identify the homotopy class of the essential curve $\gamma$, the geodesics $\gamma$ in this class, and the corresponding element $\gamma\in\CM\CL(S)$.

The union of non-intersecting laminations is a lamination and the sum of transverse measures on two disjoint laminations is a transverse measure on the union. It follows that the sum $\lambda_1+\lambda_2$ of two measured laminations is a well-defined measured lamination as long as the supports of $\lambda_1$ and $\lambda_2$ do not meet transversally. For example, if $\gamma_1,\dots,\gamma_r$ are the components of a simple geodesic multicurve and if $c_1,\dots, c_r$ are positive reals, then $\sum c_i\gamma_i$ is a perfectly sound measured lamination, a {\em weighted multicurve}. Along the same lines, any measured lamination $\lambda\in \CM\CL(S)$ has a unique decomposition $\lambda=\gamma + \lambda'$ where $\gamma$ is a weighted multicurve corresponding to the atomic part of $\lambda$ and where $\lambda'$ is a measured lamination whose support does not contain any closed leaves. 

We say that a measured lamination $\lambda$ is {\em ergodic} if whenever $X\subset\supp(\lambda)$ is a saturated subset of its support (that is, a union of leaves) then for every segment $I$ we have that either $\lambda^I(I\cap X)=0$ or $\lambda^I(I\setminus X)=0$. Two measured laminations $\lambda_1,\lambda_2\in\CM\CL(S)$ are {\em mutually singular} if for every segment $I$ transversal to $\lambda_\CL=\supp(\lambda_1)\cup\supp(\lambda_2)\in\CL(S)$ there is $Z\subset I$ with $\lambda_1^I(Z)=0$ and $\lambda_2^I(I\setminus Z)=0$.

Up to scaling, every lamination supports only finitely many ergodic measured laminations (recall for that matter that we give a concrete bound in  Corollary \ref{kor number of measures}). Moreover, every measured lamination can be written as the sum of ergodic ones. To be more precise, for every $\lambda\in\CM\CL(S)$ there are pairwise mutually singular ergodic measured laminations $\mu_1,\dots,\mu_r\in\CM\CL(S)$ with $\supp(\mu_i)\subset\supp(\lambda)$ and with
\begin{equation}\label{eq ergodic decomposition}
\lambda=\mu_1+\dots+\mu_r.
\end{equation}
In fact, the collection $\{\mu_1,\dots,\mu_r\}$ is unique. The measured laminations $\mu_1,\dots,\mu_r$ are the {\em ergodic components} of $\lambda$ and \eqref{eq ergodic decomposition} is its {\em ergodic decomposition}.
\medskip

Measured laminations can be seen as currents (see \cite{BonahonFrench,book} for details) and hence can be identified with measures on the unit tangent bundle $T^1S$ which are invariant under both the geodesic flow and the geodesic flip. To describe how this identification goes, suppose that $\gamma\in\CM\CL(S)$ is a closed geodesic and note that it corresponds to two orbits $\gamma_+$ and $\gamma_-$ of period $\ell_S(\gamma)$ of the geodesic flow: traversing the geodesic in one direction and in the other direction. The flip and flow invariant measure $\hat\gamma$ associated to $\gamma$ is the convex combination of the Lebesgue measure along those two orbits. In symbols, this means that for $f\in C^0(T^1S)$ we have
$$\int_{T^1S}f(v)\, d\hat\gamma(v)=\frac 12\left(\int_0^{\ell_S(\gamma)}f(\rho_t(v))\, dt+\int_0^{\ell_S(\gamma)}f(\rho_t(-v))\, dt\right)$$
where $v\in T^1S$ is any vector tangent to the geodesic $\gamma$. The factor $\frac 12$ is there to ensure that the total measure of $\hat\gamma$ agrees with the length of $\gamma$. In general we refer to the total measure
$$\ell_S(\mu)=\Vert\hat\mu\Vert$$
of $\hat\mu$ as the {\em length of $\mu\in\CM\CL(S)$}. 

\begin{bem}
We mentioned above the {\em ergodic decompostion of a measured lamination}. Well, to our surprise we found out that it is not easy to give a reference for its existence---to our surprise because it is definitively common knowledge. Let us thus briefly explain a way to get it. Given a measured lamination $\mu\in\CM\CL(S)$, let $\hat\mu$ be the associated flip and flow invariant measure. In other words, $\mu$ is a $(\BR\rtimes\BZ/2\BZ)$-invariant measure on $T^1S$. Applying the usual ergodic decomposition theorem (or maybe noting that the argument used to prove it for iterates of an individual map applies word-by-word for actions of a group such as $\BR\rtimes\BZ/2\BZ$) we get that $\hat\mu$ can be written as a linear combination of ergodic measures, or rather as an integral over the space of ergodic measures. Now, via the identification between flip and flow invariant measures and currents, we get that each one of the measures we are integrating over arises from an ergodic measured lamination carried by the support of $\mu$. This means that we can write $\mu$ as an integral over the space of ergodic measured laminations with support contained in $\supp(\mu)$. As we pointed out earlier (compare with Corollary \ref{kor number of measures}) the latter set is finite, meaning that instead of an integral we have a sum as in \eqref{eq ergodic decomposition}.
\end{bem}


The space of measures on $T^1S$ is naturally endowed with the weak-*-topology. We get thus a topology on $\CM\CL(S)$ via the embedding of the latter in the space of flip and flow invariant measures on $T^1S$. We list a few properties of the space $\CM\CL(S)$:

\begin{itemize}
\item $\CM\CL(S)$ is a metrizable space and the corresponding projective space $\CP\CM\CL(S)$ is compact. In fact, it turns out that $\CM\CL(S)$ and $\CP\CM\CL(S)$ are homeomorphic to a Euclidean space and a sphere respectively, but we will not need these facts.
\item The set of weighted closed geodesics $c\cdot\gamma$ with $c>0$ is dense in $\CM\CL(S)$.
\item The geometric intersection number of two essential curves extends to a continuous bilinear symmetric map $\iota:\CM\CL(S)\times\CM\CL(S)\to\BR_{\ge 0}$, the {\em intersection form}.
  \item A concrete distance on $\CM\CL(S)$ can be obtained as follows: it is known (see for example \cite{Hamenstadt}) that there are finite collections $\CA$ of curves which {\em separate measured laminations}, by what we mean that for any two distinct $\mu,\mu'\in\CM\CL(S)$ there is $\alpha\in\CA$ with $\iota(\mu,\alpha)\neq\iota(\mu',\alpha).$ Such a finite collection $\CA$ separating measured laminations induces a distance
\begin{equation}\label{distance CA}
d_\CA(\mu,\mu')=\max_{\alpha\in\CA}\vert\iota(\mu,\alpha)-\iota(\mu',\alpha)\vert
\end{equation}
     on $\CM\CL(S)$.
   \item The mapping class group acts on $\CM\CL(S)$ by homeomorphisms, preserving the intersection form.
   \item It is the key part of Thurston's classification of mapping classes that every pseudo-Anosov mapping class has north-south dynamics on $\CP\CM\CL(S)$. Moreover, the two fixed points $\lambda_{\pm}$ are such that $\iota(\lambda_{\pm},\gamma)>0$ for every simple curve $\gamma$. 
\end{itemize}

Although there has been no sign of it in what we said so far, la raison d'\^etre of this paper is that measured laminations behave really differently in the non-orientable compared to the orientable setting. The reason behind it all is the following theorem of Scharlemann \cite{Scharlemann}:

\begin{sat*}[Scharlemann]
Let $S$ be a non-orientable hyperbolic surface and let $\lambda\in\CM\CL(S)$ be a measured lamination in $S$. If the support of $\lambda$ contains a one-sided closed geodesic $\gamma$, then $\gamma$ is also contained in the support of every measured lamination in a neighborhood of $\lambda$.
\end{sat*}

Since this result is of capital importance for our paper, let us sketch the proof. A first basic fact is that every two-holed projective plane $P$ contains exactly two homotopy classes of essential curves: both of them are one-sided and they intersect each other once. If we declare one of them to be {\em the core of $P$}, then the other one is {\em dual to the core}. Now, the basic observation is that any one-sided curve $\gamma$ arises as the core of some two-holed projective plane, and if $P\subset S$ is such a two-holed projective plane with core $\gamma$ and if $\eta\subset P$ is dual to $\gamma$, then a measured lamination $\lambda\in\CM\CL(S)$ has $\gamma$ as an atom if and only if 
\begin{equation}\label{eq Scharlemann}
\max_{\alpha\subset\D P}\iota(\lambda,\alpha)<\iota(\lambda,\eta)
\end{equation}
This inequality is an open condition and this proves the theorem. 

\begin{bem}
Note also that if $c\cdot\gamma$ is a one-sided atom of $\lambda$ and if $P$ and $\eta$ are as above, then $\gamma$ has weight $c=2\cdot(\iota(\lambda,\eta)-\max_{\alpha\subset\D P}\iota(\lambda,\alpha))$ and this quantity depends continuously on $\lambda$. This means that not only are one-sided atoms stable, but also their weight.
\end{bem}

Before moving on we should point out that Danthony-Nogueira \cite{Danthony-Nogueira} proved a much stronger version of Scharlemann's theorem: the set of measured laminations with a one-sided atom is not only open but also has full measure, and hence is dense.

\subsection{Non-orientable exceptional surfaces}\label{subsec exceptional}

As mentioned already, there are exactly three non-orientable surfaces with $\chi(S)=-1$, namely the two-holed projective plane (i.e. the one-holed M\"obius band), the one-holed connected sum of two projective planes (i.e. one-holed Klein bottle), and the (closed) connected sum of three projective planes (i.e. the connected sum of a torus and a M\"obius band). We denote a surface of genus $g$ and $r$ punctures as $N_{g,b}$ and so since $\chi(S)=2-g-r$, these surfaces are $N_{1, 2}$, $N_{2, 1}$ and $N_{3,0}$ respectively. Below we describe their associated spaces of measured laminations and mapping class groups. We refer to \cite{Matthieu2, Matthieu3} for more details.

We already discussed  $N_{1,2}$ above: there are only two simple closed curves $\gamma$ and $\eta$, both of which are one-sided, and $\gamma$ and $\eta$ intersect. Moreover, any non-compact simple geodesic must spiral towards either a boundary component or one of these closed curves. It follows that every measured lamination is either a multiple of $\gamma$ or a multiple of $\eta$. In particular, $\CM\CL^+(N_{1, 2})=\emptyset$ and $\CP\CM\CL(N_{1,2}) = \CP\CM\CL^-(N_{1,2})$ consists of two points. The mapping class group $\Map(N_{1,2})$ is the finite group $\mathbb{Z}/2\mathbb{Z}\times \mathbb{Z}/2\mathbb{Z}$. 

Next we consider $N_{2, 1}$, the one-holed Klein bottle. There is exactly one two-sided simple closed curve $\alpha$ and any other simple closed curve intersects it. Letting $\gamma_0$ be a one-sided closed curve, all other one-sided closed curves are obtained from $\gamma_0$ through Dehn twists along $\alpha$ and hence the set of one-sided closed curves is given by the bi-infinite sequence $(\gamma_n)$ where $\gamma_n=D^n_{\alpha}(\gamma_0)$ for each $n\in\mathbb{Z}$. Note that $\gamma_n$ is disjoint from exactly two other closed curves: $\gamma_{n-1}$ and $\gamma_{n+1}$. Moreover, any non-compact simple geodesic spirals towards one of the closed curves or around the boundary component. Consequently, any measured lamination is a simple multicurve having at most two components. In particular $\CP\CM\CL^+(N_{2,1})$ consists of one point and any $[\lambda]\in\CP\CM\CL^-(N_{2,1})$ is of the form $[t\gamma_n + (1-t)\gamma_{n+1}]$ for some $n$ and $t\in[0,1]$. Moreover, $\gamma_n$ converges projectively to $\alpha$ as $n$ tends to $\pm\infty$. So $\CP\CM\CL(N_{2, 1})$ is a circle with a marked point corresponding to $\alpha$. The mapping class group $\Map(N_{2,1})$ is $D_{\infty}\times \mathbb{Z}/2\mathbb{Z}$, where $D_{\infty}$ is the infinite dihedral group. 

\begin{bem}
We stress that neither $N_{1,2}$ nor $N_{2,1}$ carry measured laminations whose support has a non-compact leaf.
\end{bem}

Finally we deal with the connected sum of three projective planes. In this surface there exists a unique one-sided closed curve $\gamma$ disjoint from every two-sided closed curve and such that $N_{3,0}\setminus\{\gamma\}$ is orientable.  Note that $N_{3, 0}\setminus\gamma$ is a one-holed torus $T$. Hence we can embed $\CM\CL(T)$ into $\CM\CL(N_{3,0})$. Moreover, any two-sided simple closed curve is homotopic to one in $T$, and in fact, any $\lambda\in\CM\CL^+(N_{3,0})$ is contained in $T$. Any other one-sided curve intersects $\gamma$ in one point. We identify the circle $\CP\CM\CL(T)$ with  $\CP\CM\CL^+(N_{3,0})$ in the 2-sphere $\CP\CM\CL(N_{3,0})$ and it divides $\CP\CM\CL^-(N_{3,0})$ into two disks: one corresponding to measured laminations having $\gamma$ as a leaf and the other corresponding to all multicurves having a one-sided leaf other than $\gamma$ as a component. Note that since every mapping class  must fix $\gamma$ we can identify $\Map(N_{3,0})$ with $\Map(T)$, the (full) mapping class group of the one-holed torus.

\section{Uniform train tracks}\label{sec:uniform train tracks}

Throughout this section we assume that $S$ is a connected non-exceptional hyperbolic surface---we will write $\ell(\cdot)=\ell_S(\cdot)$ for lengths in $S$ of curves, arcs, and measured laminations. We assume also that $\lambda\in\CL(S)$ is a geodesic lamination without closed leaves. If the reader prefers, they can think of $\lambda$ being recurrent (that is, the support of a measured lamination) and connected, and of $S$ being non-orientable.

In the informal version of Theorem \ref{prop lm} in the introduction we said something about a sufficiently nice train track. Our goal now is to clarify what kind of train tracks we will consider and prove two technical lemmas needed in later sections. We start by recalling what train tracks are.

\subsection{Train tracks}\label{sec train track} 
A {\em train track} $\tau$ in a hyperbolic surface $S$ is a smoothly embedded graph such that every vertex is contained in the interior of a smoothly embedded arc $I\subset\tau$ and such that none of its complementary regions is a disk with at most 2 cusps, an annulus, a M\"obius band, or a disk with at most one puncture. See \cite{book}, or rather \cite{Hatcher-paper,Penner-Harer}, for details.

Except as an aid in a construction in the appendix, we will always assume that train tracks don't have superflous vertices: connected components homeomorphic to a circle have a single vertex and all other connected components have vertices which are at least trivalent. This implies that the complexity of a train track $\tau$ is bounded by the topology of the surface $S$ it lives in. We can for example bound the possible number of vertices by noting that each complementary region of $\tau$ contributes a multiple of $-\pi$ to the Gau\ss-Bonnet integrand and has at most 3 cusps:
$$\vert V(\tau)\vert = \text{ number of vertices in }\tau\le 6\vert\chi(S)\vert.$$
Since both the number of vertices and the number of edges are maximal if the train track is trivalent, we also get the following bound on the number of edges: 
$$\vert E(\tau)\vert= \text{ number of edges in }\tau\le 9\vert\chi(S)\vert.$$
These bounds imply that, from a combinatorial point of view, there are only finitely many train tracks. Since we will need this fact here and there, we record it as a lemma:

\begin{lem}\label{lem finite train tracks}
The surface $S$ has only finitely many mapping class group orbits of train tracks. \qed
  \end{lem}

A lamination $\lambda$ is {\em carried} by a train track $\tau$ if there is a continuous map
\begin{equation}\label{eq:lamination carrying homotopy}
\Psi:[0,1]\times\lambda\to S,\ \ (t,x)\mapsto\Psi_t(x)
\end{equation}
such that $\Psi_t$ is an embedding for all $t\in[0,1)$, $\Psi_0$ is the identity and $\Psi_1$ is an immersion with image contained in $\tau$. If the support of a measured lamination is carried by $\tau$ then we say that the measured lamination itself is carried by $\tau$. The set of measured laminations carried by $\tau$ is denoted by $\CM\CL(\tau)$. It is a closed subset of $\CM\CL(S)$. Also, the set of weighted multicurves is dense in $\CM\CL(\tau)$---recall that one of the main results of this paper, Theorem \ref{main theorem}, is to understand what is the closure of the set of weighted {\em two-sided} multicurves.

  Still with the same notation suppose that a measured lamination $\lambda\in\CM\CL(\tau)$ is carried by the train track $\tau$, let $e$ be an edge of $\tau$, and let $I$ be a smooth segment meeting $\tau$ transversally in a single point of $e$. The {\em weight} $\omega_\lambda(e)$ of $\lambda$ on $e$ is then the limit
$$\omega_\lambda(e)\stackrel{\rm def}=\lim_{t\to 1}\iota(\lambda,\Psi_t^{-1}(I)),$$
  where $\Psi_t$ is as in \eqref{eq:lamination carrying homotopy}. It is well-known that $\lambda\in\CM\CL(\tau)$ is uniquely determined by the vector $(\omega_\lambda(e))_{e\in E(\tau)}\in\BR_{\ge 0}^{E(\tau)}$ whose entries are the weights of the edges of $\tau$ with respect to $\lambda$. We denote by
  $$W(\tau)=\{(\omega_\lambda(e))_{e\in E(\tau)}\in\BR_{\ge 0}^{E(\tau)}\text{ with }\lambda\in\CM\CL(\tau)\}$$
the set of all the so-obtained vectors. It is well-known that the map $\CM\CL(\tau)\to W(\tau)$ is a homeomorphism and that $W(\tau)$ is the intersection of a linear subspace of $\BR^{E(\tau)}$, the set of solutions of the switch equations, with the positive quadrant $\BR_{\ge0}^{E(\tau)}$.

\subsection{Almost geodesic train tracks}\label{sec uniform}
  A train track $\tau$ in our hyperbolic surface $S$ is {\em $\epsilon_0$-geodesic} if the edges have geodesic curvature less than $\epsilon_0$ at every point. We will fix $\epsilon_0$ such that the following holds:
\begin{itemize}
\item[(*)] If $\gamma$ is any bi-infinite path in the hyperbolic plane with geodesic curvature less than $\epsilon_0$ and if $\gamma_*$ is the geodesic which stays at bounded distance of $\gamma$, then the closest point projection $\gamma_*\to\gamma$ is well-defined, moves points less than distance 1 away, and distorts length by at most a factor 2.
\end{itemize}
Fixing $\epsilon_0$ once and forever so that all of this holds, we will refer to an $\epsilon_0$-geodesic train track as an {\em almost geodesic train track}.

The reason to consider $\epsilon_0$ such that (*) holds is that if $\tau$ is an almost geodesic train track and if $\lambda$ is a lamination carried by $\tau$, then applying to each leaf the closest point projection to its image under $\Psi_1$ as in \eqref{eq:lamination carrying homotopy} we get a continuous map
\begin{equation}\label{eq: carrying map}
  \Phi=\Phi^{\lambda,\tau}:\lambda\to\tau
\end{equation}
satisfying:
\begin{itemize}
\item $d_S(\Phi(x),x)\le 1$ for all $x\in\lambda$, and
\item the restriction of $\Phi$ to each individual leaf of $\lambda$ is a smooth immersion satisfying $\frac 12\le\vert\frac{\D\Phi}{\D t}\vert\le 2$ at every point.
\end{itemize}
The map $\Phi$ as in \eqref{eq: carrying map} is the {\em carrying map} of $\lambda$ in $\tau$. With this notation we say that $\lambda$ {\em fills} $\tau$ if $\Phi$ is surjective---if this is the case we say that $\tau$ carries $\lambda$ {\em in a filling way}.
\medskip

\begin{bem}
It is well known that for any lamination $\lambda$ there is an almost geodesic train track $\tau$ carrying $\lambda$, say in a filling way, and that if $\lambda$ has no closed leaves we can find such a $\tau$ whose edges are as long as we want. There are different ways to construct such a train track, but for example one can follow Thurston \cite[Section 8.9]{Thurston-notes} and take a neighborhood of $\lambda$ and then collapse it along a foliation whose leaves are short segments transversal to $\lambda$. Another way is to create a one-vertex train track using the first return map to a short transversal of $\lambda$. We encourage the reader to think of yet other ways---there are many. 
\end{bem}

Before moving on note that if $\lambda\in\CM\CL(\tau)$ is carried by an almost geodesic train track $\tau$ then we have
\begin{equation}\label{eq combinatorial length}
\ell(\lambda)\le \sum_{e\in E(\tau)}\omega_\lambda(e)\cdot\ell(e)\le 2\cdot\ell(\lambda).
\end{equation}
We will often want to replace the actual length of the lamination by the middle term in \eqref{eq combinatorial length}---this is one of the reasons explaining why factors of 2 keep appearing in this paper.

\subsection{Uniform train tracks}\label{sec uniform2}

The chain of inequalities \eqref{eq combinatorial length} probably makes apparent that almost geodesic train tracks are geometric objects---objects that reflect the geometry of $S$. It is often however very comfortable to work with train tracks as if they were purely combinatorial objects. In some sense, {\em uniform train tracks} are the kind of train tracks where these two points of view are married. In a nutshell, a uniform train track is an almost geodesic train track all of whose edges have comparable lengths. The actual definition is a bit more subtle because we are using $\lambda$ to measure the length of edges. That is why we first need some terminology.

With notation as above, suppose that $\tau$ is an almost geodesic train track carrying our lamination $\lambda$, and let $\Phi:\lambda\to\tau$ be the carrying map. A pair $\{e,e'\}$ consisting of two half-edges of $\tau$ based at the same vertex is said to be a {\em turn} of $\tau$. A turn is {\em $\lambda$-legal} if it is represented by $\Phi(I)$ for some smooth segment $I\subset\lambda$. An edge $e\in E(\tau)$ of $\tau$ is {\em $\lambda$-loopy} if both its half-edges are based at the same vertex and form a $\lambda$-legal turn---note that a $\lambda$-loopy edge is necessarily two-sided, see Figure \ref{pic loopy}. Finally we will say that $\tau$ is {\em $\lambda$-generic} if it has no bivalent vertices, if all its vertices have at most valence 4, and if moreover all valence 4 vertices correspond to $\lambda$-loopy edges.

\begin{figure}[h]
\centering
\includegraphics[width=0.8\textwidth]{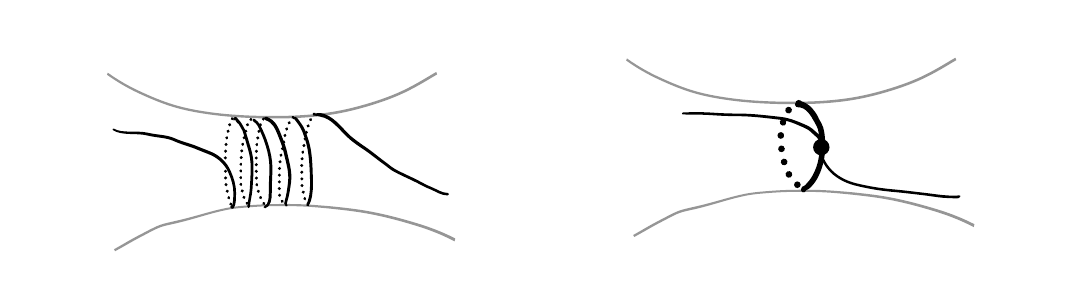}
\caption{On the left a local picture of a lamination $\lambda$; on the right a local picture of a train track $\tau$ carrying $\lambda$. The thick edge is a $\lambda$-loopy edge of $\tau$}
\label{pic loopy}
\end{figure}

\begin{bem}
  If the lamination $\lambda$ is understood from the context then we might simply say that a turn is {\em legal} instead of $\lambda$-legal, that an edge is {\em loopy} instead of $\lambda$-loopy, and that the train track is question is {\em generic} instead of $\lambda$-generic and so on. We stress once again that the simple closed curve associated to a loopy edge is two-sided.
  \end{bem}

Still with the same notation, we define the {\em $\lambda$-length} of a closed edge $e\in E(\tau)$ as follows:
\begin{equation}\label{eq talked to Hugo1}
\ell^{\lambda,\tau}(e)=\max_{I\in\pi_0(\Phi^{-1}(e))}\ell(\Phi(I)).
\end{equation}
Note that the $\lambda$-length $\ell^{\lambda,\tau}(e)$ and the usual length $\ell(e)$ agree if $e$ is not $\lambda$-loopy. It is false if $e$ is loopy---in this case $\ell^{\lambda,\tau}(e)=k\cdot\ell(e)$ for some $k\ge 2$. 

We will write $\ell^{\lambda}(e)$ instead of $\ell^{\lambda,\tau}(e)$ when the train track $\tau$ is understood from the context. Accordingly, there will be no $\tau$ in the superscript in our notation for the $\lambda$-length of $\tau$ itself and for the minimum of the $\lambda$-lengths of the edges of $\tau$:
\begin{equation}\label{eq talked to Hugo2}
\ell^{\lambda}(\tau)=\sum_{e\in E(\tau)}\ell^{\lambda,\tau}(e)\ \text{  and  }\ m^\lambda(\tau)=\min_{e\in E(\tau)}\ell^{\lambda,\tau}(e).
\end{equation}
With this notation in place, we are now ready to formally define what is a {\em $(C,\lambda)$-uniform train track}.

\begin{defi*}
  Suppose that a lamination $\lambda$ in a hyperbolic surface $S$ has no closed leaves. An almost geodesic train track $\tau$ is {\em $(C,\lambda)$-uniform} for some $C>0$ if $\lambda$ is carried by $\tau$ in a filling way and if $\ell^{\lambda}(\tau)\le C\cdot m^\lambda(\tau).$
\end{defi*}

The existence of uniform train tracks, or more specifically the fact that any train track can be refined to a uniform one, will be proved in the appendix of this paper. We say $\tau$ can be {\em refined} to a train track $\tau'$ if there is a smooth map 
$$\Psi: [0,1]\times \tau'\to S$$
such that $\Psi_t$ is an embedding for all $t\in[0,1)$ and such that $\Psi_1$ is an immersion with image $\tau$
(compare to the definition of a lamination being carried by a train track in Section \ref{sec train track}).  With slight abuse of terminology we set $\Phi^{\tau', \tau}=\Psi_1: \tau'\to \tau$ which we call a carrying map and refer to $\tau'$ as a {\em refinement} of $\tau$. With this definition in place, we state the proposition we will prove in Appendix A: 

\begin{prop}\label{prop-useful traintrack}
  Let $S$ be a hyperbolic surface. There is a constant $C>0$ such that for any geodesic lamination $\lambda\subset S$ without closed leaves, any train track $\tau_0$ carrying $\lambda$, and any $L$ large enough, there is a refinement $\tau$ of $\tau_0$ which is $\lambda$-generic, $(C,\lambda)$-uniform and satisfies $\ell^{\lambda}(\tau)\ge L$.
\end{prop}

It is maybe worth mentioning that the constant $C$ is not obtained via a compactness theorem: one can actually give concrete estimates for it. We will for example see that  one can take $C$ to be $1000\vert\chi(S)\vert$ if we drop the requirement of $\tau$ being generic and, say, $2^{30\vert\chi(S)\vert}$ if we do not. Clearly, we are not trying to obtain optimal bounds.
\medskip 

Proposition \ref{prop-useful traintrack} will be proved in Appendix A but the reader might want to at least try to come up with a proof by themselves. We would love to have simpler argument than the one we give. 

\subsection{Two lemmas on uniform train tracks}
As we mentioned earlier, uniform train tracks are geometric objects (they are almost geodesic) but they act as if they were combinatorial objects (all edges have basically the same length, meaning that up changing our measuring unit from nanometer to kyndemil we can act as if all edges had unit length). We exploit this in the proof of two technical lemmas needed below. 

The first one basically asserts that if a measured lamination carried by a uniform train track spends little of its life in some edge, then one can find another measured lamination nearby which avoids that edge altogether. The second is a very weakly quantified version of Scharlemann's theorem.

\begin{lem}\label{lem projection}
  For every geodesic lamination $\lambda$ and all positive $C$ and $\delta$ there is $\epsilon<\delta$ such that for every $(C,\lambda)$-uniform, $\lambda$-generic train track $\tau$ the following holds: For every unit length $\mu\in\CM\CL(\tau)$ there is a second unit length measured lamination $\nu\in\CM\CL(\tau)$ with
  \begin{enumerate}
    \item $\ell(e)\cdot\vert\omega_\mu(e)-\omega_\nu(e)\vert<\delta$ for all $e\in E(\tau)$, and 
    \item $\omega_\nu(e)=0$ for every $e\in E(\tau)$ with $\ell(e)\cdot \omega_\mu(e)<\epsilon$.
  \end{enumerate}
\end{lem}


We encourage the reader to find counter examples to Lemma \ref{lem projection} if one drops the assumption that $\tau$ is uniform.

\begin{proof}
To begin with let $\tau^*$ be the train track obtained from $\tau$ by removing all open $\lambda$-loopy edges and let $\tau^{**}$ be the union of the closed $\lambda$ loopy edges. In other words, $\tau^{**}$ is a disjoint union of circles, $\tau=\tau^*\cup\tau^{**}$ and $\tau^*$ and $\tau^{**}$ only meet at the vertices of $\tau^{**}$. Note also that not only do we have $\BR_{\ge 0}^{E(\tau)}=\BR_{\ge 0}^{E(\tau^*)}\oplus\BR_{\ge 0}^{E(\tau^{**})}$ but also
  $$W(\tau)=W(\tau^*)\oplus W(\tau^{**}).$$
  Said in plain terms, every solution of the switch equations of $\tau^*$ and every solution of the switch equations of $\tau^{**}$ give, when taken together, a solution to the switch equations of $\tau$, and every solution arises in this way.

  All of this means that, up to maybe solving the problem for some smaller $\delta$, we can assume that $\tau$ is either the union of disjoint circles or a trivalent train track which has thus no loopy edges. For the sake of concreteness we will assume that we are in the second case---the case that $\tau$ is a union of disjoint circles can be dealt with using exactly the same argument, but we are sure that the reader can come up with a much simpler one. 

Anyways, arguing by contradiction, suppose that the claim fails to be true for some $\lambda$, $C$ and $\delta$, meaning that there are $\epsilon_i\to 0$ such that for each $\epsilon_i$ one can find some
  \begin{itemize}
  \item $(C,\lambda)$-uniform trivalent train track $\tau_i$,
  \item and $\mu_i\in\CM\CL(\tau_i)$ with $\ell(\mu_i)=1$
  \end{itemize}
  for which we cannot find any $\nu_i\in\CM\CL(\tau_i)$ satisfying (1) and (2) in the statement for $\epsilon=\epsilon_i$.

  We get from Lemma \ref{lem finite train tracks} that the surface $S$ contains only finitely many homeomorphism types of train tracks. This means that, up to passing to a subsequence, we can assume that there is some train track $\tau$ and for each $i$ a homeomorphism $\phi_i:S\to S$ with $\phi_i(\tau)=\tau_i$. The map $\phi_i$ induces a linear isomorphism---denoted by the same symbol---between the associated spaces of solutions of the switch equations:
  $$\phi_i:W(\tau)\to W(\tau_i)$$
  Our $\mu_i$ gives us a vector $\omega_{\mu_i}=(\omega_{\mu_i}(e))_{e\in E(\tau_i)}\in W(\tau_i)$. Consider the scaled vector
  $$w_i=\ell(\tau_i)\cdot\omega_{\mu_i}$$
  and note that
\begin{align*}
\sum_{e\in E(\tau_i)} w_i(e)&=\sum_{e\in E(\tau_i)}\ell(\tau_i)\cdot\omega_{\mu_i}(e)\ge\sum_{e\in E(\tau_i)}\ell(e)\cdot\omega_{\mu_i}(e)\ge\ell(\mu_i)=1
\end{align*}
where the final inequality holds because of \eqref{eq combinatorial length}.

To get a bound in the other direction note that, since $\tau_i$ is not only $(C,\lambda)$-uniform but also trivalent and hence has no loopy edges, we have $\ell(\tau_i)=\ell^\lambda(\tau_i)\le C\cdot\ell^\lambda(e)=C\cdot\ell(e)$ for every edge $e\in E(\tau_i)$. This implies that
  \begin{align*}
  \sum_{e\in E(\tau_i)} w_i(e)
  &=\sum_{e\in E(\tau_i)}\ell(\tau_i)\cdot\omega_{\mu_i}(e)\\
  &\le C\cdot\sum_{e\in E(\tau_i)}\ell(e)\cdot\omega_{\mu_i}(e)\\
  &\le C\cdot 2\cdot\ell(\mu_i)=2\cdot C
  \end{align*}
where again the last inequality holds because of \eqref{eq combinatorial length}. Now, the chain of inequalities
$$1\le\sum_{e\in E(\tau_i)}w_i(e)\le 2\cdot C$$
implies that the elements $(\phi_i)^{-1}(w_i)\in W(\tau)$, for all $i$, belong to a compact set and that, up to passing to a subsequence, the limit
  $$w=\lim_i(\phi_i)^{-1}(w_i)\in W(\tau)$$
exists and is non-zero. Now, if we let $\hat\nu_i\in\CM\CL(\tau_i)$ be the measured lamination corresponding to $\phi_i(w)$, then $\nu_i=\frac 1{\ell(\hat\nu_i)}\hat{\nu}_i$ satisfies (1) and (2) for $\epsilon=\frac\delta 2$ and all sufficiently large $i$. This is a contradiction to our assumption---we are thus done with the proof of Lemma \ref{lem projection}.
\end{proof}

We proceed now to our (rather weak) quantification of Scharlemann's theorem:

\begin{lem}\label{lem quantified Scharlemann}
For all $C$ and $\epsilon$ positive there is $\delta<\epsilon$ such that if $\tau\subset S$ is a $(C,\lambda)$-uniform $\lambda$-generic train track, if $\mu,\nu\in\CM\CL(\tau)$ are measured laminations carried by $\tau$ with
$$\ell(e)\cdot\vert\omega_\mu(e)-\omega_\nu(e)\vert<\delta$$
for all $e\in E(\tau)$, and if $\nu$ is such that
\begin{itemize}
\item $\nu$ has unit length,
\item $\nu$ has a one-sided atom with at least weight $\frac\epsilon{\ell(\tau)}$, and
\item $\omega_\nu(e)=0$ for every loopy edge $e$, 
\end{itemize}
then $\mu$ has a one-sided atom.
\end{lem}

The key point of this lemma is that the given $\delta$ does the trick for every train track $\tau$ satisfying the stated conditions. 

\begin{proof}
Seeking a contradiction we assume that the lemma fails for some $C$ and $\epsilon$, meaning that there are
\begin{itemize}
\item a sequence $(\tau_n)$ of $(C,\lambda)$-uniform $\lambda$-generic train tracks,
\item a sequence $(\nu_n)$ of unit length measured laminations carried by $\tau_n$ with $\omega_{\nu_n}(e)=0$ for every loopy edge of $\tau_n$,
\item a sequence $(c_n\cdot\gamma_n)$ of one-sided atoms of $\nu_n$ with weight $c_n\ge\frac \epsilon{\ell(\tau_n)}$, and 
\item a sequence $(\mu_n)$ of measured laminations satisfying 
$$\ell_S(e)\cdot\vert\omega_{\mu_n}(e)-\omega_{\nu_n}(e)\vert<\frac 1n$$
for every edge $e\in E(\tau_n)$ and every $n$, and such that $\gamma_n$ is not a one-sided atom of $\mu_n$.
\end{itemize}
As in the proof of Lemma \ref{lem projection}, let us denote by $\tau_n^{**}\subset\tau_n$ the union of the loopy edges of $\tau_n$ and note that we have
\begin{equation}\label{eq I am really sick of this}
\ell(e)=\ell^{\lambda, \tau_n}(e)\ge \frac 1C\ell^{\lambda}(\tau_n)\ge \frac 1C\cdot\ell(\tau_n)
\end{equation}
for every edge $e\in E(\tau_n)\setminus E(\tau_n^{**})$. Let $K_n$ be the number of times that $\gamma_n$ meets a vertex of $\tau_n$. Since $\nu_n$, and thus $\gamma_n$, do not travel through $\tau_n^{**}$, we get that
$$1=\ell(\nu_n)\ge \ell(c_n\cdot\gamma_n)\ge \frac 12\cdot\frac\epsilon{\ell(\tau_n)}\cdot K_n\cdot\frac 1C\cdot\ell(\tau_n)$$
where the last inequality arises as follows: the factor $\frac 12$ because we are measuring lengths in the almost geodesic train track $\tau_n$, the next factor is the lower bound on the weight $c_n$, next is the number of edges that $\gamma_n$ travels through in $\tau_n$, and finally the lower bound for the length of each such edge. Anyways, cleaning up we get that 
$$K_n\le\frac {2C}\epsilon$$
is uniformly bounded. 

Now, by  Lemma \ref{lem finite train tracks} there are finitely many homeomorphism types of train tracks. Since any train track carries only finitely many curves made out of a given number of edges, we get that, up to passing to a subsequence, we can assume that there is a train track $\tau$, a sub-train track $\tau^{**}\subset\tau$, a one-sided curve $\gamma$ carried by $\tau$ but disjoint of $\tau^{**}$, and a sequence $(\phi_n)$ of homeomorphisms with 
$$\tau_n=\phi_n(\tau),\ \ \tau_n^{**}=\phi_n(\tau^{**})\text{ and }\gamma_n=\phi_n(\gamma).$$
We let $\hat\nu_n=\phi_n^{-1}(\nu_n)$ and note that $c_n\cdot\gamma$ is a one-sided atom of $\hat\nu_n$ for all $n$. This means, as \eqref{eq Scharlemann} in the proof of Scharlemann's theorem and the remark following it, that whenever $P\subset S$ is a two-holed projective plane with core curve $\gamma$ and dual curve $\eta$ we have
\begin{equation}\label{I am totally sick of this}
\iota(\hat\nu_n,\eta)=\max_{\alpha\subset\D P}\iota(\hat\nu_n,\alpha)+\frac 12\cdot c_n
\end{equation}
for all $n$. Note that $\tau^{**}$ can be seen as a two-sided multicurve because it is homeomorphic to $\tau_n^{**}$, the union of loopy edges of $\tau_n$. It follows that, since $\gamma$ is disjoint of $\tau^{**}$, we can also choose the two-holed projective plane to be contained in $S\setminus\tau^{**}$. It follows that $\D P$ and $\eta$ only meet edges in $\tau\setminus\tau^{**}$.

We claim now that the measured laminations $\hat\mu_n=\phi_n^{-1}(\mu_n)$ also satisfy \eqref{eq Scharlemann} for all sufficiently large $n$. Indeed, taking into account that the difference between the intersection numbers of $\hat\mu_n$ and $\hat\nu_n$ with $\D P$ is bounded by the product of the number of edges of $\hat\tau$ that $\D P$ meets times the maximal difference between the weights that $\hat\mu_n$ and $\hat\nu_n$ give to those edges, we get
\begin{align*}
\vert\iota(\hat\mu_n,\D P)-\iota(\hat\nu_n,\D P)\vert
&\le\text{const}\cdot\max_{e\in E(\tau)\setminus E(\tau^{**})}\vert\omega_{\hat\mu_n}(e)-\omega_{\hat\nu_n}(e)\vert\\
&\le\text{const}\cdot\max_{e\in E(\tau_n)\setminus E(\tau_n^{**})}\vert\omega_{\mu_n}(e)-\omega_{\nu_n}(e)\vert\\
&\le\text{const}\cdot\frac 1n\cdot \max_{e\in E(\tau_n)\setminus E(\tau_n^{**})}\frac1{\ell(e)}\\
&\le\text{const}\cdot\frac 1n\cdot \frac C{\ell(\tau_n)}.
\end{align*}
Here the $\max$ is taken over $E(\tau)\setminus E(\tau^{**})$ because we have chosen $P$ disjoint of $\tau^{**}$. A similar computation yields that, up to possibly replacing the constant by another one,
$$\vert\iota(\hat\mu_n,\eta)-\iota(\hat\nu_n,\eta)\vert\le\frac{\text{const}}{n\cdot\ell(\tau_n)}$$
Taking these two bounds with \eqref{I am totally sick of this}, and changing once again our constant, we get that
$$\iota(\hat\mu_n,\eta)-\max_{\alpha\subset\D P}\iota(\hat\mu_n,\alpha)\ge \frac 12\cdot c_n-\frac{\text{const}}{n\cdot\ell(\tau_n)}$$
for all $n$. Since we have the bound $c_n\ge\frac\epsilon{\ell(\tau_n)}$ we deduce that 
$$\iota(\hat\mu_n,\eta)>\max_{\alpha\subset\D P}\iota(\hat\mu_n,\alpha)$$
for all large $n$. This means that \eqref{eq Scharlemann} holds and thus that $\hat\mu_n$ has $\gamma$ as an atom. Since the pairs $(\hat\mu_n,\gamma)$ and $(\mu_n,\gamma_n)$ just differ by the homeomorphism $\phi_n$ we get that for all large $n$ the measured lamination $\mu_n$ has the one-sided curve $\gamma_n$ as an atom. This is a contradiction, and we are done.
\end{proof}

\section{Separating ergodic components}\label{sec:coloring}

Continuing with the same notation as in the previous section we assume that $S$ is a connected non-exceptional hyperbolic surface and that $\lambda\in\CL(S)$ is a connected geodesic lamination without closed leaves. We will also assume that it is recurrent. Since $\lambda$ is connected, this last assumption just means that $\lambda=\supp(\mu)$ for every non-trivial measured lamination $\mu\in\CM\CL(\lambda)$ supported by $\lambda$.

Our goal now is to prove Theorem \ref{prop lm}: 

\begin{sat}\label{prop lm}
  Let $S$ be a connected and possibly non-orientable hyperbolic surface,  let $\lambda\in\CL(S)$ be a connected recurrent lamination, and let $\tau_0$ be a train track carrying $\lambda$. Let $\mu_1,\dots,\mu_r\in\CM\CL(\lambda)$ be distinct unit length ergodic measured laminations with support $\lambda$, and let $\CA\subset S$ be a collection of simple curves separating measured laminations.

  There is $C$ such that for every $\epsilon>0$ there are
  \begin{itemize}
  \item a $(C,\lambda)$-uniform $\lambda$-generic refinement $\tau^\epsilon$ of $\tau_0$, and
  \item unit length measured laminations $\nu_1^\epsilon,\dots,\nu_r^\epsilon\in\CM\CL(\tau^\epsilon)$
  \end{itemize}
  satisfying the following:
  \begin{enumerate}
  \item $\ell_S(e)\cdot \left\vert\omega_{\mu_i}(e)-\omega_{\nu_i^\epsilon}(e)\right\vert\le\delta(\epsilon)$ for every $i=1,\dots,r$ and $e\in E(\tau^\epsilon)$.
  \item If $\tau^\epsilon_j\subset\tau^\epsilon$ is the sub-train track filled by $\nu^\epsilon_j$ then we have $\tau^\epsilon_i\cap\tau^\epsilon_j=\emptyset$ for all $i\neq j\in\{1,\dots,r\}$.
  \item We have $d_{\CA}(\mu_i^\epsilon,\mu_i)\le\delta(\epsilon)$ for all $i$, $\epsilon$, and unit length $\mu_i^\epsilon\in\CM\CL(\tau_i^\epsilon)$.
  \end{enumerate}
Here $\delta(\cdot)$ is a function with $\lim_{\epsilon\to 0}\delta(\epsilon)=0$ and $d_{\CA}$ is as in \eqref{distance CA}.
\end{sat}

Given that the statement of Theorem \ref{prop lm} is rather technical, the reader might be wondering what it actually means. It might help to re-visit the informal version that we stated in the introduction and compare with the one here, but we also hope that its content is clearer after we have proved the following result:

\begin{kor}\label{kor lm}
Let $\lambda\in\CL(S)$ be a lamination and $\mu_1,\dots,\mu_r\in\CM\CL(\lambda)$ be distinct unit length ergodic measured laminations with support $\lambda$. There is a sequence of weighted simple multicurves $(\gamma^i)$ with $\gamma^i=c^i_1\cdot\gamma^i_1+\dots+c^i_r\cdot\gamma^i_r$ such that $\mu_j=\lim_{i\to\infty}c^i_j\cdot\gamma^i_j$ for all $j=1,\dots,r$. 
\end{kor}

Corollary \ref{kor lm} is due to Lenzhen-Masur \cite[Theorem C]{Lenzhen-Masur} in the orientable case.

\begin{proof}
We will just prove the claim under the assumption that $\lambda$ is connected and recurrent, leaving to the reader to make the necessary changes to deal with the disconnected case. These assumptions imply that $\lambda=\supp(\mu_j)$ for all $j$. Let also $\CA$ be as in the statement of Theorem \ref{prop lm} and, for $\epsilon\to 0$ let $\tau^\epsilon$ and $\tau_1^\epsilon,\dots,\tau_r^\epsilon$ be the produced train tracks. Let then $c_j^\epsilon\cdot\gamma_j^\epsilon\in\CM\CL(\tau_j^\epsilon)$ be a unit length weighted curve carried by $\tau^\epsilon_j$---as we mention earlier the set of weighted curves is dense, and in particular non-emtpy, in the set of measured laminations carried by a train track. Anyways, we have
$$\mu_j=\lim_{\epsilon\to 0}c^\epsilon_j\cdot\gamma^\epsilon_j$$
because of (3) in the theorem. Moreover, $c^\epsilon_1\cdot\gamma^\epsilon_1+\dots+c^\epsilon_r\cdot\gamma^\epsilon_r$ is a simple multicurve because, by (2) in the theorem, we have $\tau^i_\epsilon\cap\tau^j_\epsilon=\emptyset$ for all $i\neq j$. We are done.
\end{proof}

\subsection{Proof of Theorem \ref{prop lm}}

We turn now to the proof of Theorem \ref{prop lm}. We will complete it first modulo the following result that we deal with below. It basicall says that, as long as the shortest edge in the train track has sufficiently long $\lambda$-length, then the measured laminations $\mu_1,\dots\mu_r$ spend most of their time in different parts of the train track. More precisely, if we write
\begin{equation}\label{eq set of heavy edges}
E(\tau,\mu,\epsilon)=\{e\in E(\tau)\text{ with }\ell(e)\cdot\omega_{\mu}(e)>\epsilon\}
\end{equation}
then the proposition says that for all $\epsilon$ the sets $E(\tau,\mu_i,\epsilon)$ and $E(\tau,\mu_j,\epsilon)$ are disjoint for suitably chosen $\tau$ and $\mu_i\neq\mu_j$:

\begin{prop}\label{prop different colors}
  Suppose that $\lambda,\mu_1,\dots\,\mu_r$ and $\CA$ are as in Theorem \ref{prop lm}. For every $\epsilon$ there is $L_\epsilon$ such that for every almost geodesic train track $\tau$ carrying $\lambda$ and with $m^\lambda(\tau)\ge L_\epsilon$ the following holds:
  \begin{itemize}
  \item We have $E(\tau,\mu_i,\epsilon)\cap E(\tau,\mu_j,\epsilon)=\emptyset$ for all $i\neq j$.
    \item Moreover, any unit length $\nu_i\in\CM\CL(\tau)$ with
  $$\{e\in E(\tau)\text{ with }\omega_{\nu_i}(e)\neq 0\}\subset E(\tau,\mu_i,\epsilon)$$
      satisfies $d_\CA(\nu_i,\mu_i)<\delta(\epsilon)$ with $\lim_{\epsilon\to 0}\delta(\epsilon)=0$.
  \end{itemize}
\end{prop}

Proposition \ref{prop different colors} is basically a property of mutually singular ergodic measures invariant under some dynamical system. The main difference between Theorem \ref{prop lm} and Proposition \ref{prop different colors} is that the latter does not assert anything about the topology of the sets of edges $E(\tau,\mu_i,\epsilon)$---for all we know, this set could consist of a single edge, meaning that there is no non-trivial measured lamination satisfying the condition in the second statement of the proposition. We work with uniform train tracks exactly to get some control of these sets. This is were Lemma \ref{lem projection} will come handy.

Assuming Proposition \ref{prop different colors} for the time being, we prove Theorem \ref{prop lm}:

\begin{proof}[Proof of Theorem \ref{prop lm}]
First note that if $\lambda$ has a closed leaf then we actually have that $\lambda$ is itself a circle and hence that the theorem is trivially true. We suppose from now on that $\lambda$ has no closed leaf. 

  Let $C$ be the constant provided by Proposition \ref{prop-useful traintrack} and let $L_\epsilon$ be the constant provided by Proposition \ref{prop different colors} for our $\epsilon$. By Proposition \ref{prop-useful traintrack} there is a $(C,\lambda)$-uniform $\lambda$-generic refinement $\tau=\tau^\epsilon$ of $\tau_0$ with $\ell^{\lambda}(\tau)\ge C\cdot L_\epsilon$. Uniformity implies that $m^\lambda(\tau)\ge L_\epsilon$, and hence that Proposition \ref{prop different colors} applies. 

We start however by applying Lemma \ref{lem projection} to our $\epsilon$ and to each $\mu_i$, letting $\nu_i=\nu_i^\epsilon\in\CM\CL(\tau)$ be the obtained measured lamination. We are going to prove that these measured laminations have the properties claimed in the statement of the theorem.

  To begin with, note that (1) in the theorem holds because $\nu_i$ satisfies (1) in Lemma \ref{lem projection}. 

  To see that (2) is satisfied, note that, by Lemma \ref{lem projection}, the smallest sub-train track $\tau_i=\tau_i^\epsilon\subset\tau$ carrying $\nu_i$ is such that $E(\tau_i)\subset E(\tau,\mu_i,\epsilon)$. Now Proposition \ref{prop different colors} implies that $E(\tau_i)\cap E(\tau_j)=\emptyset$ for all $i\neq j$. This implies that $\tau_i\cap\tau_j=\emptyset$ because they are both sub-train tracks of the $\lambda$-generic train track $\tau$. We have proved that (2) in the theorem is also satisfied.

  To conclude, note that the final statement in Proposition \ref{prop different colors} implies that also (3) in the theorem is satisfied. 
\end{proof}

The remaining of this section is devoted to prove Proposition \ref{prop different colors}.

%
%
%

\subsection{Approximating segments}

Given a point $x\in\lambda$ and some $L>0$ we denote by
$$I_{x,L}^{\lambda}:[-L,L]\to\lambda$$
the path parametrized by arc length and mapping $0$ to $x$. We will not distinguish between the path $t\mapsto I_{x,L}^\lambda(t)$ and its opposite $t\mapsto I_{x,L}^\lambda(-t)$. In fact, abusing notation we often identify $I_{x,L}^\lambda$ with its image, acting as if $I_{x,L}^\lambda=I_{x,L}^\lambda([-L,L])$. We care about these segments because, for large $L$, they do a good job approximating $\lambda$. It will however be important to only work with segments $I_{x,L}^\lambda$ which are basically determined by the point $\Phi(x)$, where $\Phi$ is the carrying map of $\lambda$ in $\tau$. Note however first that $\Phi(I^\lambda_{x,L})$ and $\Phi(I^\lambda_{x',L})$ could actually be very different even if $\Phi(x)=\Phi(x')$: they might be close for quite a while, but then, upon passing a vertex, they might diverge into different edges of $\tau$. This is why we consider, for each closed edge $e\in E(\tau)$, the set
$$\BI^\lambda_L(e)=\{x\in\lambda\text{ with }\Phi(I_{x,L}^\lambda)\subset e\}$$
of those points $x$ such that $I_{x,L}^\lambda$ projects into $e$. Denote by 
$$\BI^\lambda_L(\tau)=\cup_{e\in E(\tau)}\BI_L^\lambda(e)$$
the union of all those sets. We stress that to define $\BI_L^\lambda(e)$ and thus $\BI_L^\lambda(\tau)$ we considered closed edges. This only matters, but then it matters a lot, if $e$ is a loopy edge. Anyways, the following lemma asserts that for points in $\BI_L^\lambda(\tau)$ the segment $I_{x,L}^\lambda$ is to all effects determined by $x$.

\begin{lem}\label{lem same intersection}
  For any finite set of closed geodesics $\CA$ there is a constant $c(\CA)$ such that if $\tau$ is an almost geodesic train track and if $\lambda$ is a lamination carried by $\tau$, then we have
  $$\left\vert\vert\alpha\cap I^\lambda_{x,L}\vert-\vert\alpha\cap I^\lambda_{x',L}\vert\right\vert\le c(\CA)$$
  for any $\alpha\in\CA$ and any two $x,x'\in\BI_L^\lambda(\tau)$ with $\Phi(x)=\Phi(x')$.
\end{lem}
\begin{proof}
There is $d_0$ such that whenever we have $\Phi(x)=\Phi(x')$ for $x,x'\in\BI^\lambda_L(\tau)$, then (up to reversing one of the segments if necessary) there is a homotopy from $I_{x,L}^\lambda$ to $I_{x',L}^\lambda$ whose tracks have length at most $d_0$---for example $d_0=6$ does the trick. Now, for any given closed geodesic $\alpha$ there is a bound $c(\alpha)$ for the maximal number of times that it can meet a geodesic segment of length $d_0$ in general position. This quantity, or rather its double, bounds then the difference between the number of times that $\alpha$ can meet $I_{x,L}^\lambda$ and $I_{x',L}^\lambda$ if $x,x'\in\BI^\lambda_L(\tau)$ are such that $\Phi(x)=\Phi(x')$. The claim follows when we take $c(\CA)=2\cdot\max_{\alpha\in\CA}c(\alpha)$.
\end{proof}

Recall now that any measured lamination, and thus any unit length $\mu\in\CM\CL(\lambda)$ can be interpreted as a flip and flow invariant probability measure $\hat\mu$ on $T^1S$. Let then $\bar\mu$ be the push-forward of $\hat\mu$ under the projection map $T^1S\to S$. We refer to $\bar\mu$ as {\em the probability measure on $S$ associated to the unit length lamination} $\mu$. 

We will want to make sure that we arrange things in such a way that $\BI_L^\lambda(\tau)$ has large $\bar\mu$-measure. Indeed, if one fixes the train track $\tau$ first and chooses $L$ second it might well be that $\BI^\lambda_L(\tau)$ is empty---it makes more sense to do things the other way around:

\begin{lem}\label{lem most of it is good}
For all $L_1$ and $\epsilon$ there is $L_2$ such that the following holds: If $\mu\in\CM\CL(\lambda)$ is a unit length measured lamination supported by $\lambda$ and if $\tau$ is an almost geodesic train track with $m^\lambda(\tau)>L_2$ then
$$\bar\mu(\BI_{L_1}^\lambda(\tau))\ge 1-\epsilon$$
where $\bar\mu$ is the probability measure on $S$ associated to $\mu$.
\end{lem}

\begin{proof}
With some $L_2$ to be determined below, suppose that $\tau$ is almost geodesic with $m^\lambda(\tau)>L_2$ and let $\Phi:\lambda\to\tau$ be the corresponding carrying map. Note that to get the desired bound it suffices to prove that
\begin{equation}\label{Rammstein playing}
\bar\mu(\Phi^{-1}(e)\setminus\BI_{L_1}^\lambda(e))<\epsilon\cdot\bar\mu(\Phi^{-1}(e))
\end{equation}
for all $e\in E(\tau)$.

The set $\Phi^{-1}(e)$ is the disjoint union of segments. If $e$ is not loopy then $\Phi$ maps each such segment $s$ homeomorphically onto $e$ and we thus have $\ell(\Phi(s))=\ell^{\lambda}(e)$ and hence that $\ell(s)\ge \frac 12\cdot\ell(e)>\frac 12\cdot L_2$ because $\Phi$ distorts distances by at most a factor of $2$. If $e$ is loopy then we can either have $\ell(\Phi(s))=\ell^\lambda(e)$ or $\ell(\Phi(s))=\ell^{\lambda}(e)-\ell(e)$ because some of the arcs may turn around $e$ once less than others. In any case, we get from the definition of $\lambda$-loopy that $\ell^{\lambda}(e)\ge 2\cdot\ell(e)$ for every $\lambda$-loopy edge $e$. This means that we still have
$$\ell(\Phi(s))\ge \frac 14\cdot L_2$$
for any component $s\in\pi_0(\Phi^{-1}(e))$, independently of $e$ being loopy or not. 

On the other hand we have for each such segment $s\in\pi_0(\Phi^{-1}(e))$ that 
$$\ell(s\setminus\BI_L^\lambda(e))\le 4\cdot L_1+4$$
because $s$ has two end-points, because $\Phi$ displaces points at most distance 1 and distorts distances by at most a factor $2$. 

This means that if we take $L_2>(16\cdot L_1+16)\cdot\epsilon^{-1}$ we have that
$$\ell(s\setminus\BI_L^\lambda(e))\le\epsilon\cdot\ell(s).$$
Since this holds for every $s\in\pi_0(\Phi^{-1}(e))$ we get  \eqref{Rammstein playing}, as we wanted.
\end{proof}

To conclude with these generalities note that if $e\in E(\tau)$ is an edge in an almost geodesic train track carrying $\lambda$ then we have $\bar\mu(\Phi^{-1}(e))\ge \frac 12\cdot\ell(e)\cdot\omega_\mu(e)$ because each segment in $\Phi^{-1}(e)$ has at least length $\frac 12\cdot\ell(e)$ and there are $\omega_\mu(e)$ worth of segments. If $e$ is loopy this is still true when we just cut $e$ open at a point (equivalently, when we work with open edges). We record a minimal generalisation of this fact:

\begin{lem}\label{lem mass of and edge}
If $\tau$ is an almost geodesic train track carrying $\lambda$ and $\mu\in\CM\CL(\lambda)$ is a unit length measured lamination supported by $\lambda$ then we have
$$\bar\mu(\Phi^{-1}(U))\ge \frac 12\cdot\ell(U)\cdot\omega_\mu(e)$$
for every edge $e\in E(\tau)$ and every measurable subset $U\subset e$. Here $\bar\mu$ is the probability measure on $S$ associated to $\mu$ and $\Phi:\lambda\to\tau$ is the carrying map.\qed
\end{lem}

All of this concerns arbitrary measured laminations with support $\lambda$---things become much more interesting when we restrict ourselves to ergodic measured laminations.

\subsection{The Birkhoff game}

The segments $I_{x,L}^\lambda$ considered above will be important because they do not only approximate $\lambda$, but also, if $x$ is well-chosen, they approximate concrete measured laminations supported by $\lambda$. Indeed, applying the Birkhoff ergodic theorem, we get that for every $L^1$-function $f$ the limit
$$\lim_{L\to\infty}\frac 1{2L}\int_{-L}^L(f\circ I_{x,L})dt$$
of time averages exists for $\bar\mu$-almost every $x\in\supp(\mu)$. Moreover, if $\mu$ is ergodic then we get that this limit of time averages actually agrees with the space average $\int fd\bar\mu$. The convergence in the ergodic theorem is only pointwise, but we get from Egorov's theorem that it is uniform as long as we are willing to forgo a set of $x$'s with small but positive measure. In this way we get:

\begin{lem}\label{lem be}
  Let $\mu\in\CM\CL(\lambda)$ be a unit length ergodic measured lamination supported by $\lambda$ and let $\CA$ be a finite collection of closed geodesics. For every $\epsilon,\delta$ there are a compact set $X_{\mu,\CA,\delta,\epsilon}\subset\lambda$ with $\bar\mu(X_{\mu,\CA,\delta,\epsilon})\ge 1-\epsilon$ and some $L_0\ge 0$ with 
    $$\max_{\alpha\in\CA}\left\vert\iota(\mu,\alpha)-\frac 1{2L}\vert I_{x,L}^\lambda\cap\alpha\vert\right\vert\le\delta$$
for all $x\in X_{\mu,\CA,\delta,\epsilon}$ and $L\ge L_0$.\qed
\end{lem}

With the same notation as above choose some $\epsilon>0$, let $\CA$ be a finite collection of closed geodesics and let $\mu,\mu'\in\CM\CL(\lambda)$ be unit length ergodic measured laminations with support $\lambda$. Note that,  as long as the intersection numbers $\iota(\mu,\alpha)$ and $\iota(\mu',\alpha)$ are distinct enough for some $\alpha\in\CA$, we have that the sets $X_{\mu,\CA,\delta,\epsilon}$ and $X_{\mu',\CA,\delta,\epsilon}$ provided by Lemma \ref{lem be} are disjoint. Indeed, we get from Lemma \ref{lem same intersection} and Lemma \ref{lem be} that for every $\delta>0$ with
  $$\delta<\frac 12\cdot d_\CA(\mu,\mu')$$,
there is $L_1>0$ such that for every almost geodesic train track $\tau$ carrying $\lambda$ we have 
\begin{equation}\label{lem now it is motorhead}
 \Phi\big(X_{\mu,\CA,\delta,\epsilon}\cap \BI^\lambda_{L_1}(\tau)\big)\cap \Phi\big(X_{\mu,\CA,\delta,\epsilon}\cap\BI^\lambda_{L_1}(\tau)\big)=\emptyset.
\end{equation}
Recall that $d_\CA(\cdot,\cdot)$ is the distance induced by $\CA$---see \eqref{distance CA}.

Let now $L_2$ be the constant provided by Lemma \ref{lem most of it is good} for our $L_1$ and $\epsilon$, and assume that $m^\lambda(\tau)>L_2$. Then we get from Lemma \ref{lem be} and Lemma \ref{lem most of it is good}
\begin{align*}
\bar\mu(X_{\mu,\CA,\delta,\epsilon}\cap\BI_{L_1}^\lambda(e))&\ge\bar\mu(\Phi^{-1}(e))-2\cdot\epsilon\\
\bar\mu'(X_{\mu',\CA,\delta,\epsilon}\cap\BI_{L_1}^\lambda(e))&\ge\bar\mu'(\Phi^{-1}(e))-2\cdot\epsilon
\end{align*}
for every edge $e\in E(\tau)$. Combining these two bounds with Lemma \ref{lem mass of and edge} we get that 
\begin{align*}
\ell(e\setminus\Phi(X_{\mu,\CA,\delta,\epsilon}\cap\BI_{L_1}^\lambda(e)))&\ge \frac{4\cdot\epsilon}{\omega_\mu(e)}\\
\ell(e\setminus\Phi(X_{\mu',\CA,\delta,\epsilon}\cap\BI_{L_1}^\lambda(e)))&\ge \frac{4\cdot\epsilon}{\omega_{\mu'}(e)}
\end{align*}
From \eqref{lem now it is motorhead} we thus get that at most for one of $\mu$ and $\mu'$ we can have that the right side is larger than $\frac 12\cdot\ell(e)$, meaning that we have the implication
$$\ell(e)\cdot\omega_\mu(e)>8\cdot\epsilon\ \Rightarrow\ \ell(e)\cdot\omega_{\mu'}(e)<8\cdot\epsilon$$
and the same if we revers the roles of $\mu$ and $\mu'$. With $E(\cdot,\cdot,\cdot)$ as in \eqref{eq set of heavy edges}, this reads as
$$E(\tau,\mu,8\cdot\epsilon)\cap E(\tau,\mu',8\cdot\epsilon)=\emptyset.$$
We summarize what we got so far:

\begin{lem}\label{lem almost done}
  Let $\lambda\in\CL(S)$ be a lamination and $\CA$ a finite set of closed geodesics separating measured laminations. Let also $\mu_1$ and $\mu_2\in\CM\CL(\lambda)$ be distinct, unit length, ergodic measured laminations with support $\lambda$. Then for every $\epsilon$ there is $L_2$ such that whenever $\tau$ is an almost geodesic train track carrying $\lambda$ and with $m^{\lambda}(\tau)\ge L_2$ then $E(\tau,\mu_1,\epsilon)\cap E(\tau,\mu_2,\epsilon)=\emptyset$. \qed
\end{lem}

We are now ready to conclude the proof of Proposition \ref{prop different colors}.

\subsection{Proof of Proposition \ref{prop different colors}}

  From Lemma \ref{lem almost done} we get that there is some $L_\epsilon$ such that for any almost geodesic train track $\tau$ carrying $\lambda$ and with $m^{\lambda}(\tau)\ge L_\epsilon$ for all $e\in E(\tau)$ we have
    $$E(\tau,\mu_i,\epsilon)\cap E(\tau,\mu_j,\epsilon)=\emptyset$$
  for all $i\neq j\in\{1,\dots,r\}$. This was the first claim of the proposition.
  Suppose now that we have $\nu_1\in\CM\CL(\tau)$ with
\begin{equation}\label{when is this finished??}
 \{e\in E(\tau)\text{ with }\omega_{\nu_i}(e)\neq 0\}\subset E(\tau,\mu_i,\epsilon).
\end{equation}
  For $\alpha\in\CA$ we can then compute
  $$\iota(\nu_1,\alpha)=\frac 1{2L_0}\int\vert I^{\nu_1}_{x,L_0}\cap\alpha\vert\ d\bar\nu_1(x),$$
  where $L_0=L(\mu,\CA,\delta)$ is as in Lemma \ref{lem be}, and where we abused notation slightly by writing $I^{\nu_1}_{x,L_0}$ instead of $I^{\supp(\nu_1)}_{x,L_0}$. Anyways, since $\supp(\nu_i)$ is carried by the almost geodesics train track $\tau$ we get as in the argument leading to Lemma \ref{lem same intersection} that $\vert I^{\nu_1}_{x,L_0}\cap\alpha\vert$ is basically the same as $\vert I^{\lambda}_{y,L_0}\cap\alpha\vert$ for any $y\in\lambda\cap\BI^{\lambda}_{L_0}(\tau)$ with $\Phi(y)=\Phi(x)$: 
$$ \frac 1{2L_0}\left\vert\vert I^{\nu_1}_{x,L_0}\cap\alpha\vert-\vert I^{\lambda}_{y,L_0}\cap\alpha\vert\right\vert<\delta+\frac{c(\CA)}{L_0}$$
with $c(\CA)$ as in Lemma \ref{lem same intersection}. 

Combining assumption \eqref{when is this finished??} and Lemma \ref{lem be} we get that for most  $x\in\supp(\nu_1)$ there is $y_x\in X_{\mu_1,\CA,\delta,\epsilon}$ with $\Phi(x)=\Phi(y_x)$. Now, if the train track $\tau$ is such that $m^{\lambda}(\tau)$ is sufficiently large, we get from Lemma \ref{lem most of it is good} that for most $x$ the corresponding $y_x$ belongs to $\BI^{\lambda}_{L_0}(\tau)$, say for all $x$ outside a set of $\nu_1$-measure $\rho(L)$ where $\lim_{L\to 0}\rho(L)=0$.

  On the other hand there is a constant which just depends on $\CA$ and the geometry of the underlying surface, bounding the number of intersections that a segment of length $L_0$ can have with $\CA$. We get altogether there is some $\text{const}>0$ with
  \begin{align*}
    \vert\iota(\mu_1,\alpha)-\iota(\nu_1,\alpha)\vert
    &\le \frac 1{2L_0}\int\big\vert\vert I^{\nu_1}_{x,L_0}\cap\alpha\vert-\vert I^{\lambda}_{y_x,L_0}\cap\alpha\vert\big\vert\ d\bar{\nu}_1(x)\\
    &\le \left(\delta+\frac{c(\CA)}{L_0}\right)(1-\rho(L))+\text{const}\cdot\rho(L).
  \end{align*}
  This means that, if we wanted to guarantee that $\nu_1$ and $\mu_1$ have basically the same intersection number with $\alpha$, first we should take $L_0=L(\mu_1,\CA,\delta)$ large as in Lemma \ref{lem be} and then we should make $L$ huge and assume that our train track $\tau$ satisfies $m^{\lambda}(\tau)\ge L$. We are done.\qed

\section{Approximating by two-sided curves}\label{sec:two-sided}


In this section we will prove Theorem \ref{main theorem} from the introduction, the central result of this paper. The main tool is Theorem \ref{prop lm} from the previous section: we will show that, when $\lambda$ has no one-sided leaves, each one of the train tracks $\tau_j^{\epsilon}$ given by that theorem carries a two-sided curve. To do so we will show that if a train track does not carry a two-sided closed curve, then the only measured laminations it does carry are one-sided multicurves which are not too long. We then conclude the proof using Lemma \ref{lem quantified Scharlemann} to argue that the measured laminations $\nu_i^{\epsilon}$ given by Theorem \ref{prop lm} are not such multicurves. In any case, we start by investigating which train tracks carry two-sided curves. 

\subsection{Train tracks without two-sided curves} Our goal here is to describe which train tracks fail to carry any two-sided curves. For expedience, and since it will suffice for our purpose, we will only consider one-vertex train tracks. 

\begin{prop}\label{prop one-vertex tt}
Let $\tau$ be a one-vertex train track. Either $\tau$ carries a two-sided simple closed curve or every measured lamination carried by $\tau$ is a one-sided multicurve, each of whose components passes through the vertex at most twice. 
\end{prop} 

Recall that we only apply the adjectives one-sided and two-sided to curves which are essential---this means in particular that orientation preserving curves bounding M\"obius bands do not count as two-sided.
\medskip

Starting with the discussion of the proposition, consider $\tau$ a one-vertex train track. Since it will help to visualize what we describe here,  we will arbitrary choose to denote the outgoing half-edges at $v$ as \emph{top} half-edges and the incoming as \emph{bottom} half-edges.
We will have three types of edges: 
\begin{enumerate}
\item t--b edges, formed by a top half-edge and a bottom half-edge, 
\item t--t edges, formed by two top half-edges, 
\item b--b edges, formed by two bottom half-edges. 
\end{enumerate}

Since $\tau$ only has one vertex, every closed edge $e$ is a loop and hence determines a homotopy class of a simple closed curve, which we call the {\em curve associated to $e$}---note that this curve is carried by $e$ if and only if $e$ is of form t--b. Moreover, the {\em curve associated to a pair} of edges of type t--t and b--b is the homotopy class of the simple closed curve obtained as the concatenation of the two edges. Finally, we say an edge is two-sided (resp. one-sided) if the curve associated to it is two-sided (resp. one-sided). 

%

So far we have been thinking of train tracks as graphs, but it is well known that one can think of them as being made out of bands (see for example \cite{Otalbook}). In our concrete setting we represent our single vertex by a relatively wide and low rectangle that we imagine as a switch board. Accordingly, we take a long and thin band for each edge, looking a bit like a flat cable, and we plug the two short sides of the band somewhere at the top and bottom sides of out switchboard, possibly both ends on the top, or bottom, or one end on the top and the other on the bottom. We do that in such a way that at the end of the day the 2-complex that we obtain can be embedded into the surface as a regular neighborhood of our train track---in the language of \cite{book} what we have gotten is a {\em thickening} of the train track. See the left and center pictures in Figure \ref{fig:bands} for a schematic representation of the train track as a graph and as a band complex. The rightmost image in Figure \ref{fig:bands} is the set of instructions that one needs to follow to recover the train track: plug in the bands into the switchboard as indicated in the instructions manual, the two ends of the band labeled ``a" in the positions labeled by ``a" with the orientation indicated by the arrows. Note in particular that for t--b edges, two arrows in the same direction tell us that the edge associated is two-sided, while two arrows going in different directions tell us that the edge associated is one-sided; for t--t and b--b edges it is the other way around.

\begin{figure}[h]
\centering
\includegraphics[width=1\textwidth]{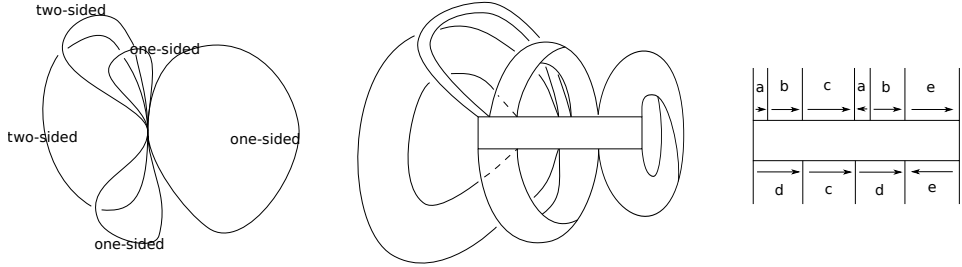}
\caption{On the left, a one-vertex train track as a graph. As a band complex in the middle. And on the right, the instructions manual indicating how the bands have to be plugged into the switchboard to build the band complex.  \label{fig:bands}} 
\end{figure}

Suggesting that the reader keeps Figure \ref{fig:defin} in mind, we will now look at some possible mutual positions of edges. 

\begin{defi}
We say that 
\begin{enumerate}
\item two t--b edges are \emph{crossing} if the two corresponding curves $\gamma$ and $\gamma'$ have intersection number $i(\gamma, \gamma')=1$,
\item a t--b edge $e$ is \emph{separated} from a pair of edges $e_t$ of type t--t and $e_b$ of type b--b if the curve $\gamma$ corresponding to $e$ and the curve $\gamma'$ corresponding to the pair $e_t, e_b$ have $i(\gamma, \gamma')=0$,
\item two edges of type t--t (resp. b--b) are \emph{crossing} if the corresponding curves $\gamma$ and $\gamma'$ have intersection number $i(\gamma, \gamma')=1$,
\item two edges $e$ and $e'$ of type t--t (resp. b--b) which are not crossing, can be either \emph{nested} or \emph{separated} according to whether, for any fixed b--b (resp. t--t) edge $e_0$, the curves $\gamma$ and $\gamma'$ associated to the pairs $e,e_0$ and $e',e_0$ have $i(\gamma, \gamma')=0$ (then they are nested) or $i(\gamma, \gamma')=2$ (then they are separated).
\end{enumerate}

\begin{figure}[h]
\centering
\includegraphics[width=1\textwidth]{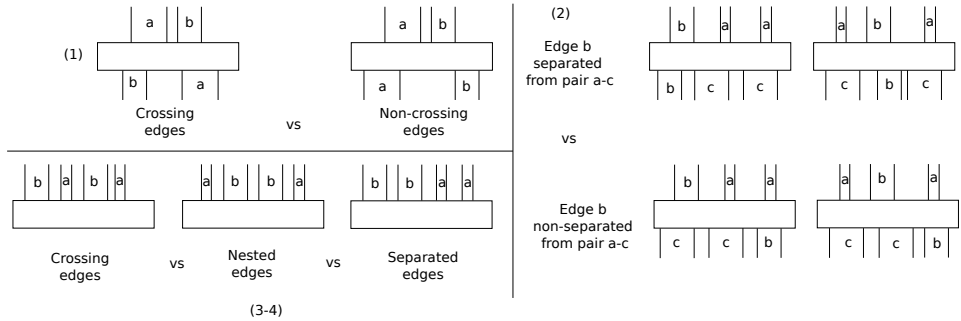}
\caption{Some examples of crossing and non-crossing t--b edges (1), t--b edges separated and non-separated from a t--t, b--b pair (2), crossing, separated and nested t--t edges (3-4). Note that there are no arrows because the definitions are the same whether the edges are one- or two-sided. \label{fig:defin}}
\end{figure}

\end{defi}

We are now ready to characterize those one-vertex train tracks that do not carry any two-sided simple closed curve. 

\begin{lem}\label{lem 7 or curve}
Let $\tau$ be a one-vertex recurrent train track. 
If any of the following statements fails, then $\tau$ carries a two-sided simple closed curve. 
\begin{enumerate}
\item All edges of type t--b are one-sided. 
\item All edges of type t--t are two-sided (resp. one-sided), while all edges of type b--b are one-sided (resp. two-sided). 
\item All t--b edges are non-crossing. 
\item All t--b edges are separated from all pairs of t--t and b--b edges.
\item If there are two or more t--t edges, then they are crossing (resp. nested).
\item If there are two or more b--b edges, then they are nested (resp. crossing). 
\item There is at most one t--t edge or at most one b--b edge. 
\end{enumerate}
\end{lem}

Recall that a train track is recurrent if it carries a curve which gives positive weight to every edge. 

\begin{proof}
First note that if (1) fails, meaning that there is a t--b edge which is two-sided, then the corresponding curve is a two-sided simple closed curve carried by $\tau$. We can suppose from now on that (1) holds true.

Similarly, if there is a pair of edges of types t--t and b--b which are both one-sided or both two-sided, then the curve corresponding to the pair is a two-sided simple closed curve. 
This means that if (2) fails then $\tau$ carries a two-sided curve and we are done. We can assume from now on that (2) holds true. But not only that. We can assume, up to swapping the roles of top and bottom, that all the t--t edges are two-sided and all the b--b edges are one-sided. 

Now assume that (3) fails, and hence that there is a pair of t--b edges which are crossing. Then one can construct a two-sided simple closed curve running along both of them as shown in Figure \ref{fig:fail123}. Ergo, we can also assume that (3) holds.

\begin{figure}[h]
\centering
\includegraphics[width=0.7\textwidth]{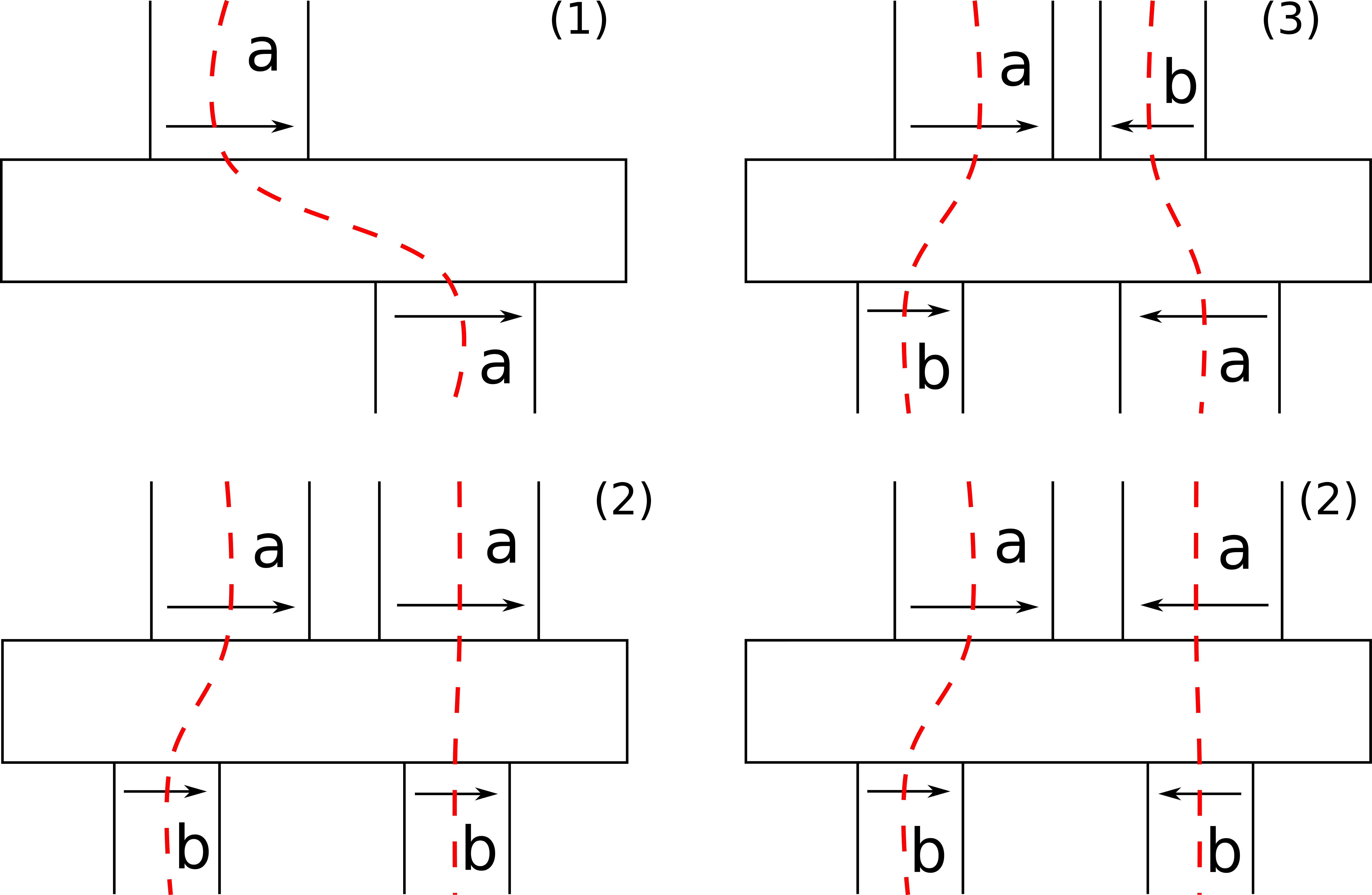}
\caption{The two-sided simple closed curves carried by $\tau$ if (1), (2) or (3) in Lemma \ref{lem 7 or curve} fail. \label{fig:fail123}}
\end{figure}

If there is a t--b edge which is non-separated from a pair of t--t and b--b edges, then one can find a two-sided simple closed curve passing once through each edge. Figure \ref{fig:fail4} shows an example, and all other possible configurations can be obtained from the same figure by flipping the picture horizontally or vertically. We can thus assume that (4) holds true.

\begin{figure}[ht]
\centering
\includegraphics[width=.7\textwidth]{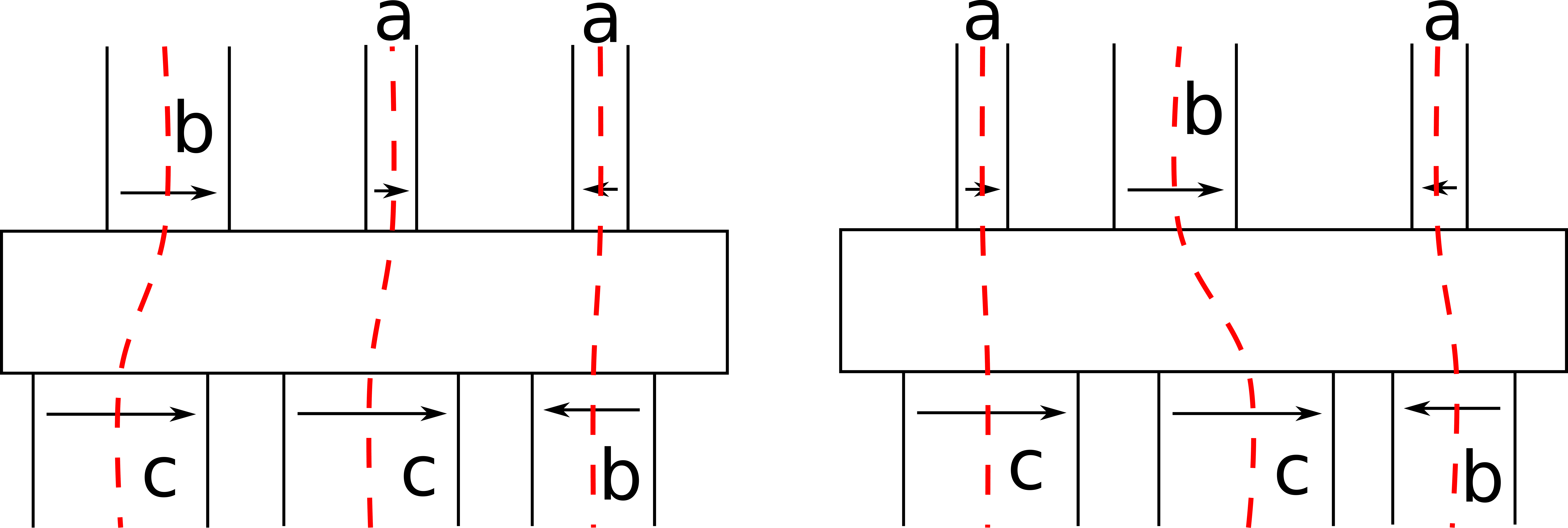}
\caption{The two-sided simple closed curve carried by $\tau$ if (4) in Lemma \ref{lem 7 or curve} fails. \label{fig:fail4}}
\end{figure}

Assume now that there are two (two-sided) t--t edges that are either nested or separated. Recurrence implies that there must also be at least one (one-sided) b--b edge. As indicated by the top half of Figure \ref{fig:fail56}, we get a two-sided simple closed curve which passes once along each t--t edge and twice along the b--b edge. This means that we can suppose that also (5) holds. 

Similarly, if (6) fails we get a simple two-sided curve as in the bottom half of Figure \ref{fig:fail56}. We can thus assume that (6) holds.

\begin{figure}[h]
\centering
\includegraphics[width=.7\textwidth]{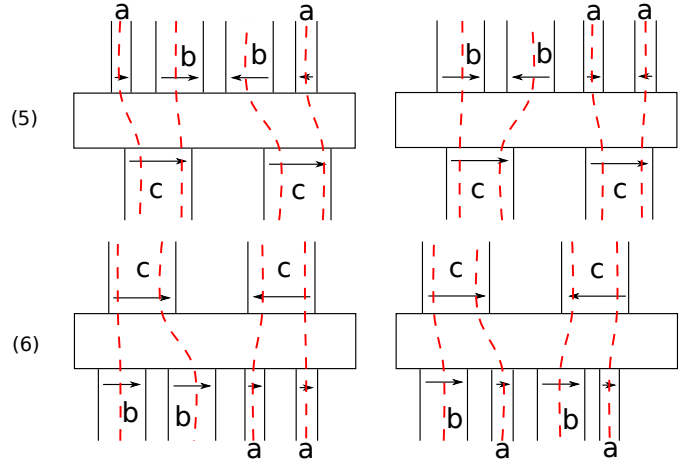}
\caption{The two-sided curve carried by $\tau$ if either (5) or (6) in Lemma \ref{lem 7 or curve} fails.\label{fig:fail56}}
\end{figure}

Finally, assume that (7) fails, that is that we have two t--t edges and two b--b edges and note that by the above we can assume that the t--t are crossing and the b--b are nested. Then we can construct a two-sided simple closed curve passing once through each edge as in Figure \ref{fig:fail7}.

\begin{figure}[ht]
\centering
\includegraphics[width=0.4\textwidth]{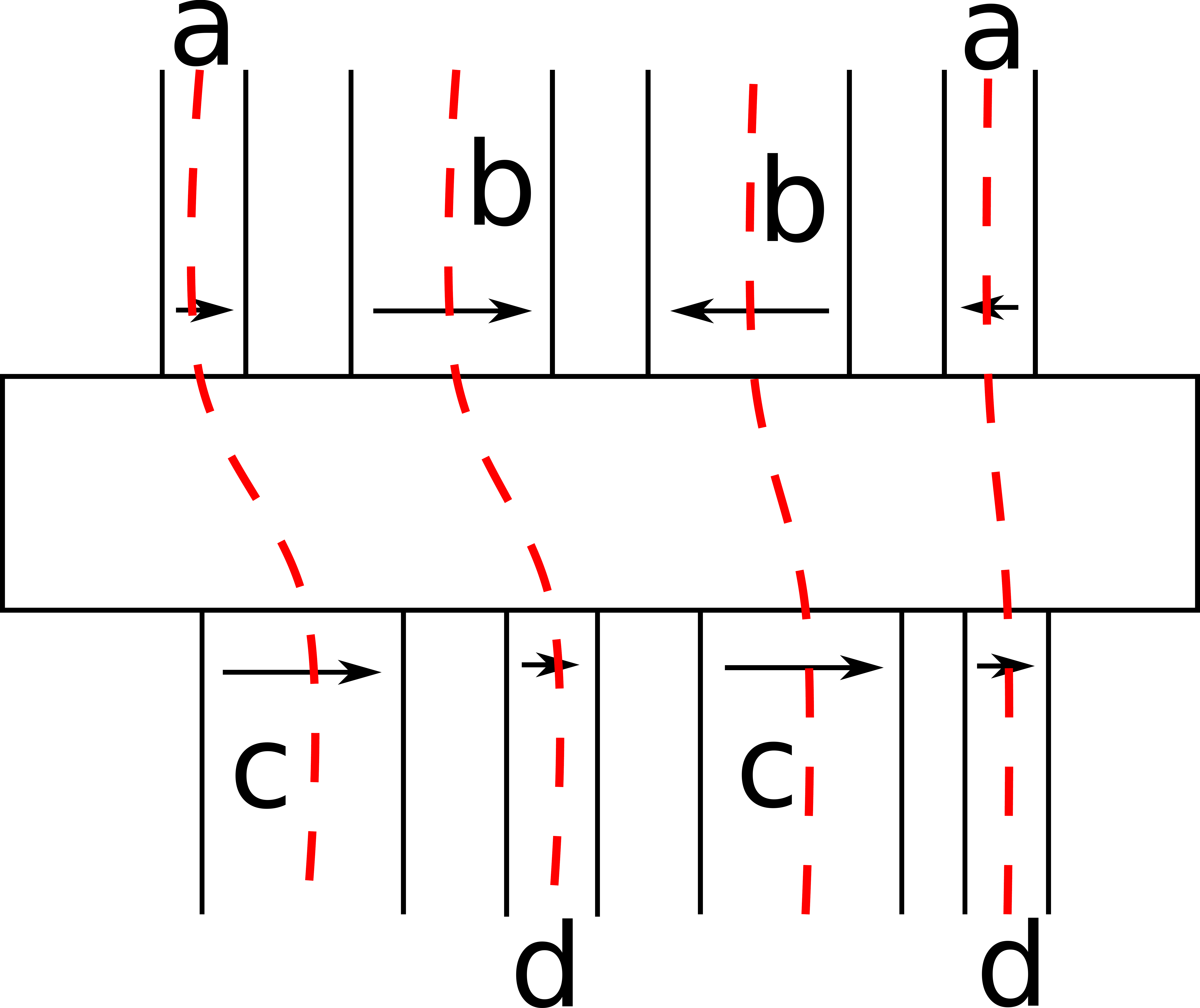}
\caption{The two-sided curve carried by $\tau$ if (7) in Lemma \ref{lem 7 or curve} fails.\label{fig:fail7}}
\end{figure}

As we wanted, we have found our two-sided curve if any one of the conditions (1)--(7) fails to be true.
\end{proof}

Our next goal is to figure out which measured laminations are carried by recurrent train tracks $\tau$ for which conditions (1)--(7) in Lemma \ref{lem 7 or curve} hold true.

Let us start by dissecting such a train track $\tau$. From (2) we know that, up to switching the roles of top and bottom, we can assume that all t--t edges are two-sided and that all b--b edges are one-sided. Moreover, from (7) we know that there can be either at most a single t--t edge or at most a single b--b edge. The discussions of both main cases (there is a single t--t edge or a single b--b edge) are very similar and the case that such edges do not exist is just a degenerate version of the main cases. To avoid repeating ourselves we will thus assume that there is a single b--b edge, leaving the other cases to the reader.

Summing up, we are assuming that there is one and only one b--b edge which is moreover one-sided (see Figure \ref{fig:1to7}, part (a)).  

Now recurrence implies that there is also at least one t--t edge: if no such t--t edge exists then all solutions of the switch equations give $0$-value to our b--b edge. Now, every t--t edge is two-sided by (2), but we might have more that one such edge. However, by (5), each pair of t--t edges must be crossing, so their mutual positions need to be like in Figure \ref{fig:1to7}, part (b).

Now we need to see where the possible t--b edges could be plugged in. First of all, if there are any t--b edges, then they are one-sided by (1).  Now, by (4), they need to be separated from each t--t, b--b pair. This means that they can be in either of three possible positions with respect to the other edges, as in Figure \ref{fig:1to7}, part (c). Finally, there could be more than one t--b edges in any of these three positions, but by (3) they need to be mutually non-crossing. See Figure \ref{fig:1to7} part (d) for the schematic description of our train track. (Figure \ref{fig:1to7} part (e) is a schematic description of the train track if there is instead a single t--t edge but maybe many b--b edges). 


\begin{figure}[t]
\centering
\includegraphics[width=\textwidth]{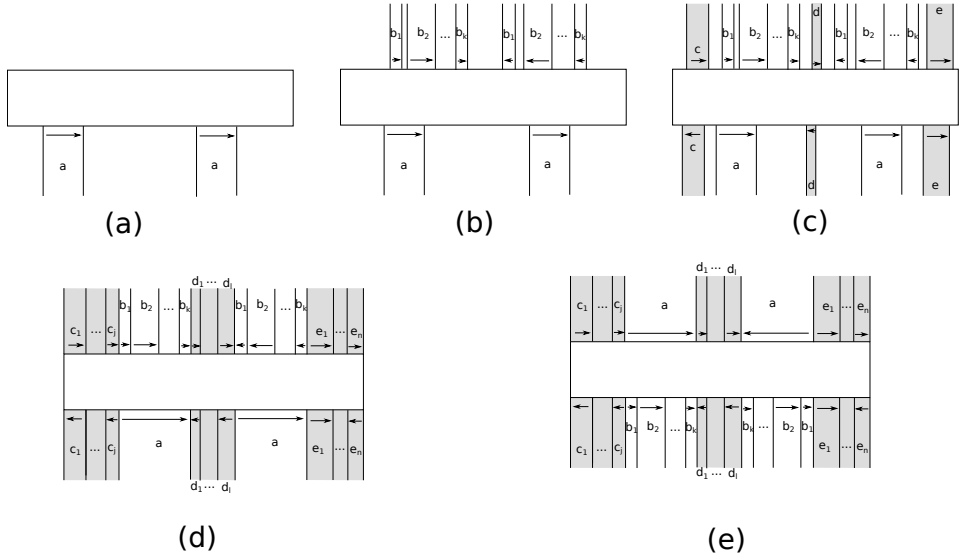}
\caption{The instruction manual of a train track satisfying (1)--(7) in Lemma \ref{lem 7 or curve} in the "single b--b edge" model: first plug in the unique b--b edge (a), then plug in the t--t edges making sure that they are crossing (b), and finally start plugging in the t--b edges making sure that are separated from every t--t, b--b pair (c). The final product is represented in (d). Finally, (e) is a representation of the final product if what you have is the "single t--t edge" model instead.\label{fig:1to7}}
\end{figure}

\medskip

Once that we understand how a train track $\tau$ satisfying (1)--(7) in Lemma \ref{lem 7 or curve} looks like, let us figure out what are the measured laminations carried by $\tau$. Given a solution $(\omega(e))_{e\in E(\tau)}$ to the switch equation at the single vertex $v$ we can represent each edge $e$ by a band of thickness $\omega(e)$. The band corresponding to the vertex then has thickness equal to the sum of the weights of all outgoing (top) half-edges, which is by definition the same as the sum of the weights of all the incoming (bottom) half-edges. 

First note that the t--b edges give equal contributions to the incoming and outgoing components.
Then, to satisfy the switch equation, one needs to have that the sum of the weights of the t--t edges must be equal to the weight of the single b--b edge.

\begin{figure}[t]
\centering
\includegraphics[width=.7\textwidth]{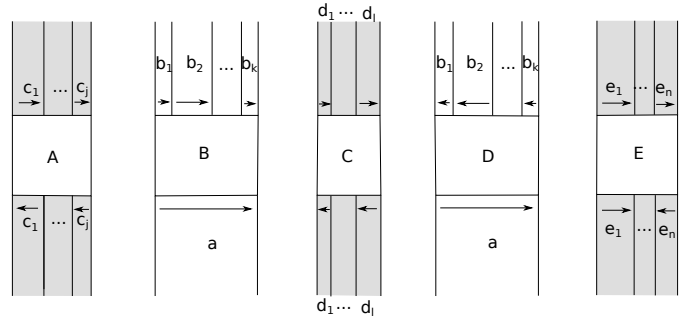}
\caption{Splitting the train track $\tau$ represented by Figure \ref{fig:1to7} part (d) into 4 disjoint train tracks. \label{fig:division}}
\end{figure}

This means that we can subdivide the rectangle representing the vertex into five (possibly degenerate) sub-rectangles (see Figure \ref{fig:division}). 
In each rectangle the sum of the weights of the edges at the top is equal to the sum of the weights of the edges at the bottom. 
The first rectangle (from the left), $A$, has attached to it the top and bottom extremes of non-crossing t--b edges. 
The second rectangle $B$ has attached to it one extreme of the single one-sided edge at the bottom and one extreme of each of the crossing two-sided edges at the top. 
The third rectangle $C$ has the same non-crossing t--b edges attached to both the top and the bottom. 
The fourth rectangle $D$ has the other extremes of the edges attached to $B$. 
The fifth rectangle $E$ has again the same non-crossing t--b edges attached to both the top and the bottom. 

In order to see which measured laminations are carried by $\tau$, inside each sub-rectangle we need to match each interval at the top (resp. bottom) with an interval at the bottom (resp. top) of same length, without intersections. 
Inside $A$, the set of intervals at the top match exactly the ones at the bottom in length. 
This means that we can only connect a t--b band to itself. 
Then a measured lamination $\lambda$ carried by $\tau$ can be written as $\lambda = c_1\gamma_1 + \dots c_a \gamma_a + \lambda'$, where $a$ is the number of t--b bands attached to $A$, $\gamma_i$ is the one-sided curve associated to each one of these bands and $\lambda'$ is carried by the sub-train track obtained removing $A$.
Note that each component $\gamma_i$ passes through the vertex only once.
The exact same thing can be done with $E$. 

Inside $B$, the sum of the weights at the top is equal to the weight of the single band at the bottom. 
So we need to subdivide the bottom interval in $b$ subintervals, where $b$ is the number of crossing two-sided bands at the top of the same length of the ones at the top. 
Then we match the interval in the only possible way to not intersect. 
A similar thing needs to be done in $D$. 
Because of the way the bands attached to $B$ and $D$ are connected, the components of $\lambda$ coming from these rectangles can only be one-sided weighted curves, each of which passes through the vertex exactly twice. 

Finally, inside $C$, the intervals are again matching because there are only t--b edges. 
This means that the components given by $C$ are again weighted one-sided curves passing through the vertex exactly once. 
\medskip

Summing up, we have proved that if a recurrent one-vertex train track $\tau$ satisfies conditions (1)--(7) in Lemma \ref{lem 7 or curve} then every measured lamination carried by $\tau$ is a one-sided weighted multicurve, each of whose components passes through the vertex at most twice. Since by Lemma \ref{lem 7 or curve} any recurrent one-vertex train track for which any one of the conditions (1)--(7) fails carries a two-sided curve, we have proved Proposition \ref{prop one-vertex tt}. \qed

\subsection{Density of two-sided curves in $\CM\CL^+(S)$}

We can now prove Theorem \ref{main theorem} from the introduction, stating that $\CP\CM\CL^+(S)$ is exactly the closure of the set of two-sided curves in $\CP\CM\CL(S)$: 

\begin{named}{Theorem \ref{main theorem}}
  Let $S$ be a connected, possibly non-orientable, non-exceptional hyperbolic surface of finite topological type. The set of two-sided curves is dense in $\CP\CM\CL^+(S)$.
\end{named}

Proposition \ref{prop one-vertex tt} will play a key role in the proof of the above theorem. In particular, in the course of the proof we will have to pass from dealing with a given train track to a one-vertex train track. Let us explain how to do this. Suppose $\tau$ is a train track carrying a measured lamination $\nu$ and let $e$ be an edge with maximal weight. Take a short transversal $I$ to $e$, think of it as a transversal to $\nu$, and consider the first return map. Identify arcs that are homotopic in $(S, I)$ and then collapse $I$ to a point $v$. The result is a refinement $\hat{\tau}$ of $\tau$ with a single vertex $v$.

\begin{proof}
Let $\mu\in\CM\CL^+(S)$ be the measured lamination that we want to approximate by weighted simple two-sided curves. Denote by $\lambda=\supp(\mu)$ the support of $\mu$ and let $\tau$ be a train track carrying $\lambda$ in a filling way. We will approximate $\mu$ by two-sided multicurves carried by $\tau$. This means in particular that we can assume without loss of generality that $\lambda$ is connected. Noting that if $\lambda$ is a closed curve then it is two-sided because we are assuming that $\mu\in\CM\CL^+(S)$, we see that the only case we truly need to consider is that $\lambda$ is connected but is not a closed curve---meaning in particular that $\lambda$ has no compact leaves. This is the setting from now on.

Well, let then $\mu=\sum_{j=1}^n a_j\mu_j$ be the ergodic decomposition of $\mu$, with $\mu_j\in\CM\CL(\lambda)$ normalized in such a way that all of the $\mu_j$ have unit length, and fix also a finite collection $\CA$ of curves separating measured laminations. For $\epsilon$ very small---we will make it tend to $0$---let $\tau^{\epsilon}$, $\nu_j^{\epsilon}$, and $\tau_j^{\epsilon}$, for $j=1,\ldots, n$, be given by Theorem  \ref{prop lm}. Forgetting for now the $\nu_j^\epsilon$'s, recall that there is some uniform $C$ such that $\tau^\epsilon$ is a $(C,\lambda)$-uniform refinement of $\tau$, that $\tau_j^\epsilon$ is a sub-train track of $\tau^\epsilon$ for each $j$, that we have $\tau^\epsilon_i\cap\tau^\epsilon_j=\emptyset$ whenever $i\neq j$, and that we have 
$$d_{\CA}(\mu_j^\epsilon,\mu_j)<\delta(\epsilon)$$
for any unit length $\mu_j^\epsilon\in\CM\CL(\tau_j^\epsilon)$. In the proof of Corollary \ref{kor lm} we argued that all of this implies that if we take any curve $\gamma_j^\epsilon$ carried by $\CM\CL(\tau_j^\epsilon)$ and if we chose weights $a_j^\epsilon$ wisely, then the weighted multicurves $\gamma^\epsilon=\sum_{j=1}^na_j^\epsilon\cdot\gamma_j^\epsilon$ converge to $\mu$. What is different now is that we have to argue that we can choose the curves $\gamma_j^\epsilon$ to be two-sided.

This is where the $\nu_j^\epsilon$ come into the game, but first note that any loopy edge is two-sided---this means that we can assume that our train tracks $\tau_j^\epsilon$ do not have loopy edges. From Theorem \ref{prop lm} we get that $\nu_j^\epsilon$ has unit length, is carried by $\tau^\epsilon_0$, and satisfies 
$$\ell(e)\cdot\vert \omega_{\mu_j}(e) - \omega_{\nu_j^\epsilon}(e)\vert<\delta(\epsilon)$$
where $\delta(\epsilon)\to 0$. The assumption that $\tau_j^\epsilon$ has no loopy edges implies that $\omega_{\nu_i^\epsilon}(e)=0$ for every loopy $e\in E(\tau^\epsilon)$, and hence that Lemma \ref{lem quantified Scharlemann} applies. Since $\mu_j$ is supported by $\lambda$, we know for a fact that it has no atom, and a fortiori no one-sided atom. We get thus from Lemma \ref{lem quantified Scharlemann} that any one-sided atom $c_j^\epsilon\cdot\gamma_j^\epsilon$ of $\nu_j^\epsilon$ has to satisfy that
\begin{equation}\label{eq I am enjoying this}
\lim_{\epsilon\to 0} c_j^\epsilon\cdot\ell(\tau^\epsilon)=0.
\end{equation}
We are going to argue that this implies that $\tau^\epsilon_j$ carries a two-sided curve, as we wanted to know. 

Seeking a contradiction, suppose that this is not the case and let $e$ be an edge of $\tau^{\epsilon}_j$ with maximal possible weight $\omega_{\nu_j}^{\epsilon}(e)$. Since $\nu_j^{\epsilon}$ is unit length, since $\tau_j^{\epsilon}$ is uniform without loopy edges, and since there is a universal bound on how many edges it can have, we have 
\begin{equation}\label{eq talked to pepe}
\omega_{\nu_j^\epsilon}(e)\ge\frac{\const}{m^\lambda(\tau^\epsilon_j)}\ge\frac{\const}{m^\lambda(\tau^\epsilon)}
\end{equation}
where $m^\lambda(\cdot)$ is as in \eqref{eq talked to Hugo2}. Now refine $\tau^{\epsilon}_j$ to a one-vertex train track $\hat{\tau}^{\epsilon}_j$ as explained above, using a transversal to the edge $e$ and the first return map of $\nu_j^{\epsilon}$. Noting that any two-sided curve in $\hat\tau_j^\epsilon$ is also carried by $\tau_j^\epsilon$ we get from our assumption and from Proposition \ref{prop one-vertex tt} that every measured lamination carried by $\hat\tau_j^\epsilon$ is a one-sided multicurve whose components pass through the vertex at most twice. This applies in particular to the largest sublamination $\sum_k c_k^{\epsilon,j}\cdot\gamma_k^{\epsilon,j}$ of $\nu_j^\epsilon$ carried by $\hat\tau_j^\epsilon$ (the whole of $\nu^\epsilon_j$ if $\tau_j^\epsilon$ is connected). Noting that each one of the one-sided curves $\gamma_k^{\epsilon,j}$ pass through the vertex of $\hat\tau_j$ at most twice, we get that $\sum_k c_k^{\epsilon,j}\le 2\cdot \omega_{\nu_j^\epsilon}(e)$. Since the number of components is bounded we have a constant such that there is always some $k$ with $c_k^{\epsilon,j}\ge\const\cdot\omega_{\nu_j^\epsilon}(e)$. Up to replacing the constant by another constant we get from \eqref{eq talked to pepe} that for all $\epsilon$ and $j$ the measured lamination $\nu_j^\epsilon$ has an atom of at least weight $c_k^{\epsilon,j}\ge\frac{\const}{m^\lambda(\tau^\epsilon)}.$ Since $\tau^\epsilon$ is $(C,\lambda)$-uniform we also get, up to again replacing the constant by yet another one, that $\nu_j^\epsilon$ has an atom whose weight satisfies 
$$c_k^{\epsilon,j}\ge\frac{\const}{\ell(\tau^\epsilon)}.$$
This contradicts \eqref{eq I am enjoying this}. 
\end{proof}

Before moving on we point out that Theorem \ref{main theorem}, or rather its proof, yields a bound for the number of mutually singular ergodic measured laminations with given support. 

Indeed, suppose that $\lambda$ is a recurrent lamination in $S$ and let $\mu_1,\dots,\mu_n\in\CM\CL(S)$ be mutually singular ergodic measured laminations with support $\lambda$. By Corollary \ref{kor lm} we can find a sequence of simple multicurves $c^1_i\cdot\gamma^1_i+\dots+c^n_i\cdot\gamma^n_i$ such that 
\begin{equation}\label{eq levitt}
\mu_j=\lim_{i\to\infty}c_i^j\cdot\gamma^j_i.
\end{equation}
It follows directly that $n$ is bounded by the maximal cardinality $c(S)$ of a multicurve in $S$. Now suppose that $\lambda$ has no one-sided leaves. Then, as in the proof of Theorem \ref{main theorem} we get that we can assume that each of the curves $\gamma_i^j$ in \eqref{eq levitt} can be chosen to be two-sided. This means that in this case $n$ is bounded by the maximal cardinality $c^+(S)$ of a two-sided multicurve in $S$. We have proved the following:

\begin{kor}\label{kor number of measures}
Let $c(S)$ and $c^+(S)$ be respectively the maximal number of components of a multicurve and of a two-sided multicurve in $S$. Every lamination $\lambda\in\CL(S)$ supports at most $c(S)$ mutually singular ergodic transverse measures. Moreover, if $\lambda$ has no one-sided leaves then it supports at most $c^+(S)$ mutually singular ergodic transverse measures.\qed
\end{kor}

As we mentioned in the introduction, this result is due to Levitt \cite{Levitt} in the orientable case.

\section{Orbit closures in $\CM\CL$}\label{sec:orbit closures ML}

In this section we prove Theorem \ref{thm orbit closures ML}. We continue to assume that $S$ is a non-exceptional hyperbolic surface. As in the introduction we associate to a measured lamination $\lambda\in \CM\CL(S)$ a {\em complete pair} $(R_\lambda,\gamma_\lambda)$ as follows. Start by the decomposition $\lambda=\gamma_\lambda + \lambda'$ where $\gamma_\lambda$ is the atomic part of $\lambda$ and where $\lambda'$ has no closed leaves. Let then $R_\lambda$ be the possibly disconnected subsurface $R_{\lambda}$ of $S$ obtained by taking the union of components of $S\setminus\gamma_\gamma$ which contain a non-compact leaf of $\lambda$. We note that none of the components of $R_{\lambda}$ can be a pair of pants, a one-holed M\"obius band, or a one-holed Klein bottle---in other words, no component of $R_\lambda$ is exceptional.

As in the introduction we define
$$\gamma_\lambda+\CM\CL^+(R_\lambda)=\{\gamma_\lambda+\mu\text{ with }\mu\in\CM\CL^+(R_\lambda)\}\subset\CM\CL(S)$$
and 
$$\CG_\lambda= \cup_{\phi\in\Map(S)}\phi\big(\gamma_\lambda+\CM\CL^+(R_\lambda)\big).$$
We can now recall the statement of Theorem \ref{thm orbit closures ML}:

\begin{named}{Theorem \ref{thm orbit closures ML}}
  Let $S$ be a connected, possibly non-orientable, non-exceptional hyperbolic surface of finite topological type. We have $\overline{\Map(S)\cdot\lambda}=\CG_\lambda$ for any measured lamination $\lambda\in\CM\CL(S)$.
\end{named}

As we mentioned in the introduction, this theorem is due to Lindenstrauss-Mirzakhani \cite{Lindenstrauss-Mirzakhani} in the orientable case (see also Hamenst\"adt \cite{hamenstadt:measures}). They obtained it as a corollary of the classification of mapping class group invariant measures on $\CM\CL(S)$, always for $S$ orientable. Given that this measure classification does not exist (for good reasons) in the non-orientable case, we have to give a direct proof of the theorem. The argument bellow works both in the orientable and non-orientable case and uses only elementary facts about measured laminations and Theorem \ref{main theorem}. Theorem \ref{main theorem} is well-known when $S$ is orientable.
\medskip

Let us start with the discussion of the theorem. The set $\CG_\lambda$ is, by its very definition, mapping class group invariant and contains $\lambda$. This means that the inclusion $\overline{\Map(S)\cdot\lambda}\subset\CG_\lambda$ follows once we know that the latter set is closed. 

\begin{lem}\label{l closed}
For any $\lambda\in\CM\CL(S)$ the set $\CG_\lambda$ is closed in $\CM\CL(S)$. 
\end{lem}

\begin{proof}
First note that we get from Scharlemann's theorem \cite{Scharlemann} that the set $\CM\CL^+(R_\lambda)$ is closed in $\CM\CL(R_\lambda)$. It follows that also $\gamma_\lambda+\CM\CL^+(R_\lambda)$ is closed in $\CM\CL(S)$. This means that $\CG_\lambda$ is the union of closed sets. To prove that such a union is closed it suffices to show that any sequence $(\mu_n)\subset\CG_\lambda$ which converges in the ambient space $\CM\CL(S)$ is contained in finitely many of those closed sets, that is, in finitely many translates of $\gamma_\lambda+\CM\CL^+(R_\lambda)$. Letting thus $(\mu_n)$ be such a sequence and noting that there is a sequence of mapping classes $(\phi_n)\in\Map(S)$ with $\mu_n\in\phi_n(\gamma_\lambda+\CM\CL^+(R_\lambda))$ for all $n$, we get that $\mu_n=\gamma_n+\lambda_n$ where $\gamma_n=\phi_n(\gamma_\lambda)$ and where $\lambda_n\in\CM\CL^+(\phi_n(R_\lambda))$. 

The assumption that $\mu_n$ converges as $n\to\infty$ implies in particular that its length $\ell(\mu_n)$ is bounded. It follows that the length of $\gamma_n$ is also bounded, meaning that the multicurves $\gamma_n$ belong to a finite set. Once we know $\gamma_n$ there are only finitely many choices for $\phi_n(R_{\lambda})$ because it is a union of components of $S\setminus\gamma_n$. It follows that the sequence $(\mu_n)$ belongs to the union of finitely many translates of $\gamma_\lambda+\CM\CL^+(R_\lambda)$, as we needed to show.
\end{proof}

So far we have established that $\overline{\Map(S)\cdot\lambda}\subset\CG_\lambda$. The main step towards proving the opposite inclusion is to prove that it holds true if $\lambda$ has no atoms:

\begin{prop}\label{prop orbit closure ML no atoms}
If $S$ is a connected non-exceptional hyperbolic surface then we have $\CM\CL^+(S)\subset \overline{\Map(S)\cdot\lambda}$ whenever $\lambda\in\CM\CL(S)$ has no atoms.
\end{prop}

It is in the proof of Proposition \ref{prop orbit closure ML no atoms} that Theorem \ref{main theorem} plays a key role. 

\begin{proof}
Since the set of weighted two-sided geodesics is dense in $\CM\CL^+(S)$ by Theorem \ref{main theorem}, it is enough to prove that any $c\cdot\gamma_0$ with $c>0$ and with $\gamma_0$ two-sided is a limit of translates of $\lambda$. Our first goal is to show that we can find translates of $\lambda$ which are very close to the trivial measured lamination. To be able to make sense of words like ``close" we fix a finite collection $\CA$ of curves separating measured laminations---we might assume without loss of generality that $\gamma_0\in\CA$---and, as in \eqref{distance CA}, let $d_\CA(\cdot,\cdot)$ be the induced distance on $\CM\CL(S)$.

\begin{claim}
For all $\epsilon>0$ there exists $\lambda_{\epsilon}\in\Map(S)\cdot \lambda$ with $\iota(\lambda_\epsilon, \gamma_0)>0$ and $d_\CA(\lambda_\epsilon,0)<\epsilon$.
\end{claim}

\begin{proof}
Note that there is some $\delta(\epsilon)>0$ such that if $\mu\in\CM\CL(S)$ has length $\ell(\mu)<\delta(\epsilon)$, then
$$d_\CA(\mu,0)=\max_{\eta\in\CA}\iota(\mu,\eta)<\epsilon$$
As we see, it suffices to guarantee is that there is a $\lambda_\epsilon\in\Map(S)\cdot\lambda$ with $\iota(\lambda_\epsilon,\gamma_0)>0$ and $\ell(\lambda_\epsilon)<\delta(\epsilon)$. This is what we will do.

To begin with, note that, up to replacing $\lambda$ by a translate by a pseudo-Anosov mapping class, we can assume that $\lambda$ and $\gamma_0$ together fill the surface. Now, take a sequence $L_n\to\infty$ and for each $n$ an almost geodesic train track $\tau_n$ carrying $\lambda$ and whose edges have all at least length $L_n$---such a train track exists because $\lambda$ has no closed leaves. Since $\lambda$ and $\gamma_0$ fill, we can forget some of the members of our sequence $(\tau_n)$ and ensure that $\iota(\gamma_0,\mu)>0$ for every $\mu\in\CM\CL(\tau_n)$ and every $n$. Also, since there are only finitely many mapping class group orbits of train tracks by Lemma \ref{lem finite train tracks}, we might pass to a subsequence such that for all $n$ there is $\phi_n\in\Map(S)$ with $\phi_n(\tau_n)=\tau_1$.

Set $\lambda_n=\phi_n(\lambda)\in\CM\CL(\tau_1)$. Since $\tau_n$ is almost geodesic we have
$$L_n\cdot\omega_\lambda(e)\le 2\cdot\ell(\lambda)$$
for each edge $e\in E(\tau_n)$. This implies in turn that
\begin{align*}
\ell(\lambda_n)
&\le\sum_{e\in E(\tau_1)}\omega_{\lambda_n}(e)\cdot\ell(e)=\sum_{e\in E(\tau_1)}\omega_{\lambda}(\phi^{-1}(e))\cdot\ell(e)\\
&\le\left(\frac{2\cdot\ell(\lambda)}{L_n}\right)\cdot\ell(\tau_1)
\end{align*}.
We just have to choose $n$ with $L_n=\delta(\epsilon)^{-1}\cdot 2\cdot\ell(\lambda)\cdot\ell(\tau_1)$.
\end{proof}

Now, let $D=D_{\gamma_0}$ be a Dehn twist along $\gamma_0$. For each $\epsilon>0$ we will prove that for there exists $n_\epsilon$ so that $D^{n_\epsilon}(\lambda_\epsilon)$ converges to $c\cdot\gamma_0$ as $\epsilon$ tends to $0$, where $\lambda_\epsilon\in\Map(S)\cdot\lambda$ is provided by the claim.  Since $\CA$ is a collection of curves separating measured laminations we get that in order to make sure that $D^{n_\epsilon}(\lambda_\epsilon)$ converges to $c\cdot\gamma_0$ we just have to make sure that
\begin{equation}\label{eq I want my computer}
\lim_{\epsilon\to 0}\iota(D^{n_\epsilon}\lambda_\epsilon,\eta)=\iota(c\cdot\gamma_0,\eta)
\end{equation}
for all $\eta\in\CA$. From Lemma \ref{lem ivanov} we get that for any such $\eta\in\CA$
\begin{align*}
\iota(D^{n_\epsilon}\lambda_\epsilon,\eta)&\le n_\epsilon\cdot\iota(\lambda_\epsilon,\gamma_0)\cdot\iota(\gamma_0,\eta)+\iota(\lambda_\epsilon,\eta)\\
\iota(D^{n_\epsilon}\lambda_\epsilon,\eta)&\ge (n_\epsilon-2)\cdot\iota(\lambda_\epsilon,\gamma_0)\cdot\iota(\gamma_0,\eta)-\iota(\lambda_\epsilon,\eta)
\end{align*}.
Since $\iota(\lambda_\epsilon,\gamma_0)\neq0$ tends to $0$ when $\epsilon$ goes to $0$, we can choose $n_\epsilon\to\infty$ with 
$$\lim_{\epsilon\to 0}n_\epsilon\cdot\iota(\lambda_\epsilon,\gamma_0)=1.$$
Combining the two inequalities above, and taking into account that $\iota(\lambda_\epsilon,\eta)<\epsilon$, we get \eqref{eq I want my computer}. This concludes the proof of the proposition.
\end{proof}

We are now ready to conclude the proof of the theorem:

\begin{proof}[Proof of Theorem \ref{thm orbit closures ML}]
Since $\Map(S)\cdot\lambda\subset \CG_\lambda$, and since the latter is closed by Lemma \ref{l closed}, we have $\overline{\Map(S)\cdot\lambda} \subset \CG_\lambda$. 

For the reverse inclusion, let $\lambda=\gamma_\lambda + \lambda'$ be the unique decomposition of $\lambda$, where $\lambda'\in\CM\CL(S)$ has no atoms. Let $R_i$ be one of the connected components of $R_{\lambda}$, and let $\lambda_i'$ be a component of $\lambda'$ supported in $R_i$. As we pointed out earlier, $R_i$ is non-exceptional. By Proposition \ref{prop orbit closure ML no atoms} we have that $\CM\CL^+(R_i)\subset\overline{\Map(R_i)\cdot\lambda_i'}$. It follows thus that $\gamma_\lambda+\mu\in \overline{\Map(S)\cdot\lambda}$ for any $\mu\in\CM\CL^+(R_\lambda)$. In other words we have 
$$\gamma_\lambda+\CM\CL^+(R_\lambda)\subset\overline{\Map(S)\cdot\lambda}.$$
Since the right side is mapping class group invariant it follows that 
$$\CG_\lambda=\Map(S)\cdot(\gamma_\lambda+\CM\CL^+(R_\lambda))\subset\overline{\Map(S)\cdot\lambda},$$
as we needed to prove. 
\end{proof}

\section{Orbit closures in $\CP\CM\CL(S)$}\label{sec:orbit closures PML} 

We come now to the proofs of Theorem \ref{thm minimal set} and Theorem \ref{thm orbit closures PML}. We start by the latter:

\begin{named}{Theorem \ref{thm orbit closures PML}}
    Let $S$ be a connected, possibly non-orientable, non-exceptional hyperbolic surface of finite topological type. We have $\overline{\Map(S)\cdot\lambda}=\CP\CG_\lambda\cup\CP\CM\CL^+(S)$ for any projective measured lamination $\lambda\in\CP\CM\CL(S)$.
\end{named}

Theorem \ref{thm orbit closures PML} will follow easily once we know that the set on the right is closed. This is what we prove next:

\begin{lem}\label{lem really bad notation}
The set $\CP\CG_\lambda\cup\CP\CM\CL^+(S)$ is closed in $\CP\CM\CL(S)$ for any $\lambda\in\CP\CM\CL(S)$.
\end{lem}
\begin{proof}
Using for once square brackets to indicate projective classes, suppose that we have a sequence $([\mu_n])_n$ in $\CP\CG_\lambda\cup\CP\CM\CL^+(S)$ which converges to some $[\mu]\in\CP\CM\CL(S)$. We claim that $[\mu]$ also belongs to $\CP\CG_\lambda\cup\CP\CM\CL^+(S)$. Since $\CP\CM\CL^+(S)$ is closed, we might assume without loss of generality that our sequence is actually contained in $\CP\CG_\lambda$. This means that it is represented by a sequence $(\mu_n)\subset\CG_\lambda\subset\CM\CL(S)$. If the sequence $(\mu_n)$ converges in $\CM\CL(S)$ then we are done because $\CG_\lambda$ is closed by Lemma \ref{l closed}. We might thus assume that $\mu_n$ only converges projectively to some representant $\mu\in\CM\CL(S)$ of the class $[\mu]$, meaning that there is a sequence $\epsilon_n\to 0$ with 
$$\mu=\lim_n\epsilon_n\cdot\mu_n$$.
We claim that $\mu\in\CM\CL^+(S)$. Otherwise $\mu$ would have a one-sided atom $c\cdot\gamma$. We would then get from Scharlemann's theorem that, for all large $i$, the geodesic $\gamma$ is the support of an atom of $\epsilon_\cdot\mu_n$. In fact, the weight $c_n$ in $\epsilon_n\cdot\mu_n$ of this atom converges to $c$ and hence stays bounded from below. Since there are $\phi_n\in\Map(S)$ and $\nu_n\in\CM\CL^+(\phi_n(R_\lambda))$ with 
$$\epsilon_n\cdot\mu_n=\epsilon_n\cdot\phi_n(\gamma_\lambda)+\nu_n$$
we get that our atom $c_n\cdot\gamma$, having relatively large weight, must be an atom of $\nu_n$. This contradicts the assumption that $\nu_n$ has no one-sided leaves. We are done.
\end{proof}

We are now ready to prove the theorem:

\begin{proof}[Proof of Theorem \ref{thm orbit closures PML}]
Since $\Map(S)\cdot\lambda$ is contained in $\CP\CG_\lambda$ we get from Lemma \ref{lem really bad notation} the inclusion $\overline{\Map(S)\cdot\lambda}\subset\CG_\lambda\cup\CP\CM\CL^+(S)$. To prove the opposite inclusion note that we get from Theorem \ref{thm orbit closures ML} that $\CP\CG_{\lambda}\subset \overline{\Map(S)\cdot\lambda}$. Also, the argument used in the proof of Proposition \ref{prop orbit closure ML no atoms} implies that the projective class of every two-sided curve belongs to $\overline{\Map(S)\cdot\lambda}$. Since two-sided curves are dense in $\CP\CM\CL^+(S)$ by Theorem \ref{main theorem} we get that $\CP\CG_\lambda\cup\CP\CM\CL^+(S)\subset\overline{\Map(S)\cdot\lambda}$. We are done.
\end{proof}

Note that Theorem \ref{thm orbit closures PML} implies that $\CP\CM\CL^+(S)$ is contained in the closure of any orbit of the action $\Map(S)\actson\CP\CM\CL(S)$. Theorem \ref{thm minimal set} follows immediately:

\begin{named}{Theorem \ref{thm minimal set}}
  Let $S$ be a connected, possibly non-orientable, non-exceptional hyperbolic surface of finite topological type. The set $\CP\CM\CL^+(S)$ is the unique non-empty closed subset of $\CP\CM\CL(S)$ which is invariant and minimal under the action of $\Map(S)$.\qed
\end{named}

On the other hand, we note that $\CP\CM\CL(S)$, for $S$ non-orientable, has a dense orbit if and only if $S$ is of genus 1. Recall that the {\em genus} of the surface is defined to be the maximum number of disjoint simple curves one can cut along without disconnecting the surface.

\begin{kor}
Suppose $S$ is non-orientable. If $S$ has genus $k=1$, then $\CP\CM\CL(S)$ has a dense $\Map(S)$-orbit. If $S$ has genus $k>1$ then $\CP\CM\CL(S)$ is {\em not} the closure of any countable union of $\Map(S)$-orbits. 
\end{kor}

\begin{proof}
For $k=1$, let $\gamma$ be a one-sided curve and set $\lambda = \gamma +\lambda'$ for some $\lambda'\in\CM\CL(S\setminus\gamma)$. Note that every one-sided curve on $S$ is in the $\Map(S)$-orbit of $\gamma$, that $S\setminus\gamma$ is orientable, and that $R_{\lambda} = S\setminus\gamma$. It follows that $\CP\CG_{\lambda} = \CP\CM\CL(S)\setminus\CP\CM\CL^+(S)$ and hence that 
$$\overline{\Map(S)\cdot[\lambda]} = \CP\CG^{\lambda}\cup\CP\CM\CL^+(S) = \CP\CM\CL(S).$$

For the second statement, note that when the genus is at least $2$ then there are uncountably many projective classes of orbits of one-sided multicurves: for any 2 disjoint one-sided curves $\alpha, \beta$, the classes $[\alpha+t\beta]$ are distinct and not in the same orbit for distinct $t>0$. On the other hand, any countable union $\cup_{n\in\BN}\CP\CG^{\lambda_n}$ only contains countably many orbits of one-sided multicurves.
\end{proof}

\section{Orbit closures in Teichm\"uller space}\label{sec:orbit closures Tecih}

There are many reasons why measured laminations and projective measured laminations play a key role when studying hyperbolic surfaces, but if one were to have to choose one then one would probably mention Thurston's compactification of Teichm\"uller space \cite{Thurston-bull}. Indeed, Thurston defined a mapping class group invariant topology on
$$\overline\CT(S)=\CT(S)\cup\CP\CM\CL(S)$$
with respect to which the inclusions $\CT(S)\hookrightarrow\overline\CT(S)$ and $\CP\CM\CL(S)\hookrightarrow\overline\CT(S)$ are homeomorphisms onto their images, and where a sequence $(X_i)\subset\CT(S)$ converges to a $[\lambda]\in\CP\CM\CL(S)$ if for some, and hence any representative $\lambda\in\CM\CL(S)$ of the class $[\lambda]$, there is a sequence $(\epsilon_i)$ of positive real numbers with
$$\lim_{i\to\infty}\epsilon_i\cdot\ell_{X_i}(\gamma)=\iota(\lambda,\gamma)$$
for every essential simple closed curve $\gamma$. The space $\overline\CT(S)$ is compact when endowed with this topology---it is in fact homeomorphic to a ball, but we will not need that fact. We refer to \cite{FLP} for a discussion of Thurston's compactification in the orientable case, to \cite{Papa-Penner} for one in the non-orientable case, and to \cite{Bonahon88} for one that works either way. 

In this section we study where do orbits of the action $\Map(S)\actson\CT(S)$ accumulate in the Thurston boundary $\D\CT(S)=\CP\CM\CL(S)$.

\begin{named}{Theorem \ref{thm limit set}}
  Let $S$ be a connected, possibly non-orientable, non-exceptional hyperbolic surface of finite topological type. We have $\overline{\Map(S)\cdot X}\cap\D\CT(S)=\CP\CM\CL^+(S)$ for any point $X$ in Teichm\"uller space $\CT(S)$.
\end{named}

In the course of the proof of the theorem we will need the following fact:

\begin{lem}\label{lemma support}
Suppose $\lambda, \lambda'\in\CM\CL(S)$ are such that $\iota(\lambda, \alpha)\leq\iota(\lambda', \alpha)$ for all simple closed curves $\alpha$. Then $\supp(\lambda)\subset \supp(\lambda')$.
\end{lem}

\begin{proof}
First note that we might assume without loss of generality that $\lambda$ has connected support, meaning that there are a priori three options: either $\iota(\lambda,\lambda')>0$, or $\supp(\lambda)\cap\supp(\lambda')=\emptyset$, or $\supp(\lambda)\subset \supp(\lambda')$. To rule the first one out, let $(\gamma_i)$ be a sequence of weighted curves which converges to $\lambda'$ in $\CM\CL(S)$. By the continuity of the intersection form we have that $\lim_i\iota(\lambda',\gamma_i)=0$ while $\lim_i\iota(\lambda,\gamma_i)=\iota(\lambda,\lambda')>0$. This contradicts our assumption.

The discussion of the second case is less clean. Suppose first that the support of $\lambda$ is not a simple curve. This implies that every open neighborhood $U$ of $\supp(\lambda)$ contains a simple curve $\gamma$ with $\iota(\lambda,\gamma)>0$. If $\supp(\lambda)\cap\supp(\lambda')=\emptyset$ then we can choose the neighborhood $U$ disjoint from $\lambda'$, meaning that $\iota(\lambda',\gamma)=0$. This contradicts our assumption. It remains to consider the case when the support of $\lambda$ is a simple curve---for the sake of concreteness we assume that the weight is one. Since it is the most interesting and hardest case we suppose that the curve $\lambda$ is contained in a complementary region of $\supp(\lambda')$ homeomorphic to an annulus with cusps, leaving the other cases to the reader. Let $U$ be an open small neighborhood of $\lambda'$ and $\overline U$ its closure. Let also $[x,y]\subset S\setminus U$ be a simple arc which intersects $\supp(\lambda)$ once, and with $x,y\in\overline U$. We can now find two disjoint simple curves $\gamma_x,\gamma_y\subset\overline U$ with $x\in\gamma_x$ and $y\in\gamma_y$, and with $\iota(\lambda,\gamma_x)<1$ and $\iota(\lambda,\gamma_y)<1$. Let $\gamma$ be the simple curve obtained, up to homotopy, as the juxtaposition of 2 copies of $[x,y]$ and the curves $\gamma_x$ and $\gamma_y$. We have $\iota(\gamma,\lambda)=2$ while $\iota(\gamma,\lambda')<2$. This contradicts once again our assumption.

Having ruled out the first and second possibilities we get that $\supp(\lambda)\subset \supp(\lambda')$, as we needed to prove.
\end{proof}

We are now ready to prove the theorem.

\begin{proof}[Proof of Theorem \ref{thm limit set}]
Note that  $\CP\CM\CL^+(S)\subset \overline{\Map(S)\cdot X}\cap\D\CT(S)$ by minimality of $\CP\CM\CL^+(S)$, Theorem \ref{thm minimal set}. 

Suppose now that $[\lambda]\in\CP\CM\CL(S)$ is an accumulation point of $\Map(S)\cdot X$, say $X_n=\phi_n(X)$ with $\phi_n$ pairwise distinct. Let $\alpha$ and $\beta$ be two-sided simple closed curves which together fill $S$. Consider the function 
$$I: \CM\CL(S)\to\mathbb{R}$$
defined by 
$$\lambda \mapsto \frac{\ell_{X}(\lambda)}{\iota(\alpha, \lambda)+\iota(\beta, \lambda)}.$$
Note that $I$ is continuous, non-zero, and descends to a function on the compact space $\CP\CM\CL(S)$ and is hence bounded from above and below by positive constants. In particular, there exists $K>1$ such that 
$$\frac{1}{K}\leq \frac{\ell_{X}(\gamma)}{\iota(\alpha, \gamma)+\iota(\beta, \gamma)}\leq K$$
for all simple closed curves $\gamma$. It follows that 
\begin{equation}\label{ratio}
\frac{1}{K}\leq \frac{\ell_{\phi_n(X)}(\gamma)}{\iota(\phi_n(\alpha), \gamma)+\iota(\phi_n(\beta), \gamma)}\leq K
\end{equation}
for all $n$ and all simple closed curves $\gamma$. 

Note now that, up to reversing the roles of $\alpha$ and $\beta$ and passing to a subsequence, we can assume that $\ell_X(\phi_n(\alpha))\ge\ell_X(\phi_n(\beta))$ for all $n$. This means that if we set $\epsilon_n=\frac 1{\ell_X(\phi_n(\alpha))}$ then we have that both the sequences $\epsilon_n\cdot\phi_n(\alpha)$ and $\epsilon_n\cdot\phi_n(\beta)$ are bounded in $\CM\CL(S)$. This means that, up to possibly passing to a further subsequence, we can assume that the limits
$$\lim_{n\to\infty} \epsilon_n\cdot\phi_n(\alpha)=\mu\text{ and }\lim_{n\to\infty} \epsilon_n\cdot\phi_n(\beta)=\mu'$$
exist in $\CM\CL(S)$. Note that $\mu$ has unit length and hence is not $0$. Note also that $\mu,\mu'\in\CM\CL^+(S)$ because $\alpha$ and $\beta$ are two-sided and because $\CM\CL^+(S)$ is closed by Scharlemann's theorem.

Note that, since the mapping classes $\phi_n$ do not repeat, there is some $\gamma_0$ with $\ell_{\phi_n(X)}(\gamma_0)\to\infty$. Equation \eqref{ratio} implies then that also 
\begin{equation}\label{eq12345}
\iota(\phi_n(\alpha), \gamma_0)+\iota(\phi_n(\beta), \gamma_0)\to\infty
\end{equation}.
On the other hand 
\begin{equation}\label{eq56788}
\lim_n\epsilon_n\cdot\big(\iota(\phi_n(\alpha), \gamma_0)+\iota(\phi_n(\beta), \gamma_0)\big)=\iota(\mu,\gamma_0)+\iota(\mu',\gamma_0)<\infty
\end{equation}.
Hence equations \eqref{eq12345} and \eqref{eq56788} together imply that $\epsilon_n\to 0$.

Once we know that $\epsilon_n$ tends to $0$ we have that 
$$\iota(\mu,\mu')=\lim_n\iota(\epsilon_n\cdot\phi_n(\alpha),\epsilon_n\cdot\phi_n(\beta))=\lim\epsilon_n^2\cdot\iota(\alpha,\beta)=0$$
and hence that $\lambda'=\mu+\mu'$ is a perfectly sound measured lamination with 
\begin{equation}\label{eq tired}
\lim_n\epsilon_n\cdot\big(\iota(\phi_n(\alpha), \gamma)+\iota(\phi_n(\beta), \gamma)\big)=\iota(\lambda,\gamma)
\end{equation}
for every curve $\gamma$. Moreover, since the support of $\lambda'$ is the union of the supports of $\mu$ and $\mu'$ we have $\lambda'\in\CM\CL^+(S)$. 

Recall now that $[\lambda]\in\CP\CM\CL(S)$ is an accumulation point of the sequence $(\phi_n(X))$. From the choice of the sequence $(\epsilon_n)$ and from \eqref{ratio} we get that for all $\gamma$ the sequence $(\epsilon_n\cdot\ell_{\phi_n(X)}(\gamma))_n$ is bounded. This means that, once again up to passing to a subsequence, we might assume that the projective class $[\lambda]=\lim_n\phi_n(X)\in\overline\CT(S)$ has a representative $\lambda\in\CM\CL(S)$ with 
\begin{equation}\label{eq tired2}
\lim_n\epsilon_n\cdot\ell_{\phi_n(X)}(\gamma)=\iota(\lambda,\gamma)
\end{equation}
Combining \eqref{eq tired} and \eqref{eq tired2} we get from \eqref{ratio} that
$$\iota(\lambda,\gamma)\le \iota(K\cdot\lambda',\gamma)$$
for every curve $\gamma$. Lemma \ref{lemma support} implies thus that
$$\supp(\lambda)\subset\supp(\lambda').$$
Since $\lambda'\in\CM\CL^+(S)$ we get that also $\lambda\in\CM\CL^+(S)$, and hence that $[\lambda]=\lim_n\phi_n(X)\in\CP\CM\CL^+(S)$, as we needed to prove. We are done.
\end{proof}

\begin{appendix}

\section{Existence of uniform train tracks}\label{appendix}

In this appendix we prove Proposition \ref{prop-useful traintrack}, which states that any train track carrying a lamination without closed leaves can be refined to be arbitrarily long while also uniform. We recall the precise statement: 

\begin{named}{Proposition \ref{prop-useful traintrack}}
  Let $S$ be a hyperbolic surface. There is a constant $C>0$ such that for any geodesic lamination $\lambda\subset S$ without closed leaves, any train track $\tau_0$ carrying $\lambda$, and any $L$ large enough, there is a refinement $\tau$ of $\tau_0$ which is $\lambda$-generic, $(C,\lambda)$-uniform and satisfies $\ell^{\lambda}(\tau)\ge L$. 
  \end{named}

The reader is referred to Section \ref{sec:uniform train tracks} for the terminology and notation used in this appendix. In particular we remind the reader that $\tau$ is $(C, \lambda)$-uniform if it is almost geodesic, carries $\lambda$ in a filling way, and the $\lambda$-lengths of all edges are comparable: $\ell^{\lambda}(\tau)\leq C\cdot m^{\lambda}(\tau)$.

Note that if $\tau_0$ is any train track carrying the lamination $\lambda$, we can refine it to an almost geodesic (and as long as we want) train track $\tau_0'$, still carrying $\lambda$ (using, for example, the Thurston's construction mentioned in the remark in Section  \ref{sec uniform}). Since a refinement of a refinement of $\tau_0$ is still a refinement of $\tau_0$, we can without loss of generality assume that $\tau_0$ was already almost geodesic to begin with.  

The idea of the proof of the proposition is to modify the train track $\tau_0$ according to a series of procedures which increase the length of a short edge while controlling the total increase of length. Since the general idea is quite simple, but making it precise is tedious, we will first give an intuitive idea of the process. For simplicity assume $\tau_0$ is trivalent and has a unique shortest edge $e$. We want to refine $\tau_0$ such that this short edge disappears by getting concatenated with other edges, while only increasing the total length of the train track by a prescribed amount. There are two cases. If $e$ is the only incoming edge at vertex $v$, then there is a cusp at $v$ bounded by the two outgoing edges. We ``split" this cusp by zipping open $e$ until its other vertex: see the top part of Figure \ref{fig-first example of split}. In the resulting train track the short edge has been absorbed into the adjacent ones and the total length has increased by $\ell^{\lambda}(e)=m^{\lambda}(\tau_0)$. In the other case $e$ is one of two incoming edges at $v$. We again want to split the cusp at $v$ but this now means unzipping the outgoing edge $e'$ which could be very long. Hence, in order to control how much total length we add, if $e'$ is very long we do not split the cusp all the way to the next vertex but to some predetermined distance (say twice the minimal edge length). Formally we will do this by adding a bivalent vertex $v'$ on $e'$ at that distance and split until $v'$. See bottom part of Figure \ref{fig-first example of split}. Note that in both situations, in the local picture around $v$, we have not only gotten rid of the short edge $e$, but the two new edges have lengths greater than $2\ell^{\lambda}(e)$. One can then repeat this process finitely many times obtaining a refinement where the shortest edge length has doubled, while the length has only increased by a comparatively small amount. \\

\begin{figure}[h]
\leavevmode \SetLabels
\L(.29*.37) $v$\\%
\L(.26*.52) $e$\\%
\L(.48*.48) $\to$\\%
\endSetLabels
\begin{center}
\AffixLabels{\centerline{\includegraphics[width=0.7\textwidth]{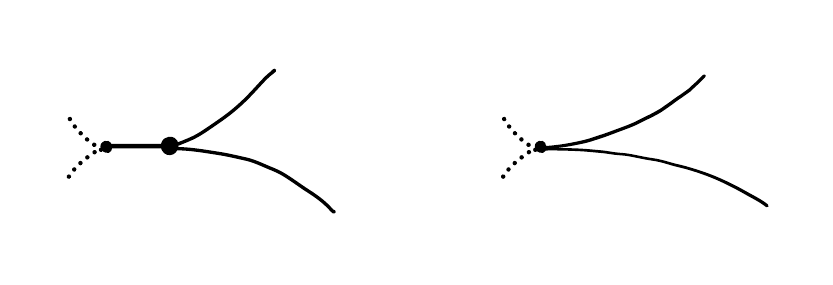}}}
\vspace{-24pt}
\end{center}
\leavevmode \SetLabels
\L(.365*.43) $v$\\%
\L(.39*.68) $e$\\%
\L(.48*.48) $\to$\\%
\L(.3*.4) $v'$\\%
\L(.25*.6) $e'$\\%
\endSetLabels
\begin{center}
\AffixLabels{\centerline{\includegraphics[width=0.7\textwidth]{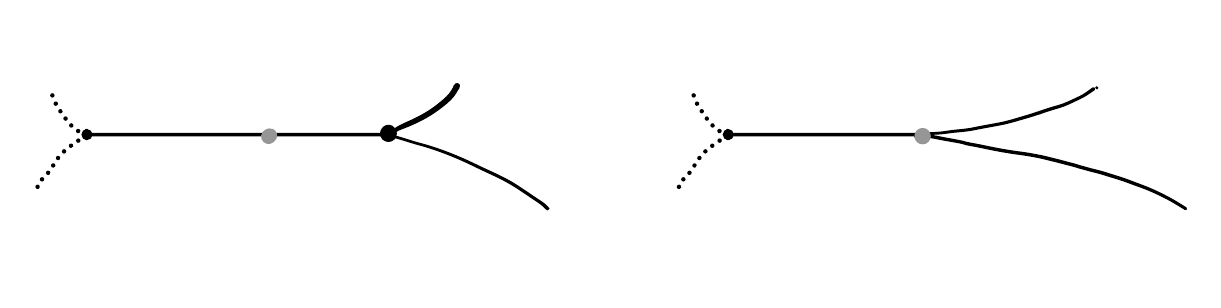}}}
\vspace{-24pt}
\end{center}
\caption{Splitting a cusp at $v$ to remove the short edge $e$.}
\label{fig-first example of split}
\end{figure}
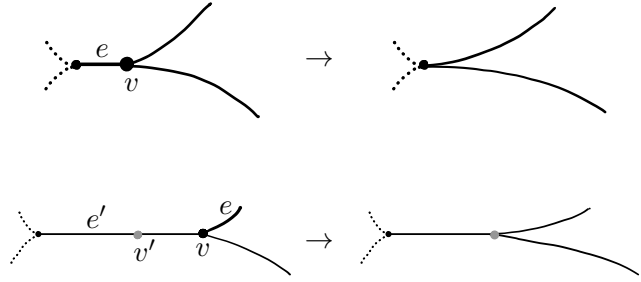

Now we will make all of this precise---the basic issue being that the process we described above works well if the involved train tracks are trivalent and that this property might get lost after even the first step. Anyways, recall from Section \ref{sec uniform2} the definition of a {\em refinement}. In this section, to give a refinement of a train track $\tau$, we will give a smooth graph $\tau'$ and a smooth immersion $\tau'\looparrowright\tau$ which can be perturbed into an embedding $\tau'\hookrightarrow S$. To be fair, we are just going to give the graph $\tau'$ and a map $\Phi=\Phi^{\tau',\tau}: \tau'\to \tau$. It will be evident that a smooth structure exists on $\tau'$ which makes $\Phi$ a smooth immersion---note that if such a structure exists, it is unique. Finally, since we are thinking of $\tau'$ in terms of an immersion into $\tau$, we identify the carrying maps of $\lambda$ into $\tau$ and $\tau'$ respectively. Accordingly, we use lengths in $\tau$ to measure the length and $\lambda$-length of edges of $\tau'$. All of this will not pose a problem because all relevant inequalities here will be strict, meaning that if they are satisfied for length defined in this way, they also hold after perturbing $\Phi$ into an embedding. Along the same lines, the condition for being almost geodesic is open and this means that it makes sense to say that $\tau'$ is almost geodesic. 


A final piece of terminology: we say that $\tau'$ is a {\em simplicial refinement} if the map $\Phi$ is a simplicial map, that is, maps vertices to vertices and edges to edges.

\medskip 

Now we define more rigorously what we mean by the process of splitting a cusp alluded to above, and in such a way that it also works for vertices of higher valence. 
Let $e$ be a (half-)edge which is not the only incoming one at a vertex $v$. Then there is a (either one or two) complementary region $F$ of $S\setminus\tau$ having a cusp at $v$ and such that $e$ is contained in a side of $F$; we call each such cusp an {\em adjacent cusp at} $v$ (with respect to $e$). Moreover, if $e'$ is an outgoing half-edge at $v$ and $\{e, e'\}$ is a $\lambda$-legal turn we say that $e'$ is a {\em $\lambda$-child} of $e$. With this terminology and with the above intuitive idea in mind, we define what we mean by splitting a cusp:

\medskip

{\bf Splitting a cusp.}  Let $F$ be a complementary region of $S\setminus\tau$ with vertex $v_0$ a cusp of $F$ and $e_0, e_1$ the two half-edges of $\tau$ incident to $v_0$ for which $v_0$ is an adjacent cusp. If $e_0$ and $e_1$ have a common $\lambda$-child denote it by $\hat{e}$. {\em Splitting the cusp $v_0$} is the process which produces the simplicial refinement $\tau'$ (possibly with bivalent vertices) of $\tau$ such that 
$$\vert\Phi^{-1}(v)\vert = 1 \text{ for all } v\neq v_0 \text{ and } \vert\Phi^{-1}(v_0)\vert = 2$$
and 
$$\vert\Phi^{-1}(e)\vert = 1 \text{ for all } e\neq \hat{e} \text{ and } \vert\Phi^{-1}(\hat{e})\vert = 2 \text{ (if $\hat{e}$ exists).}$$
Note that for every cusp of $\tau$ we split, we add the length of a common child to the total length of the resulting train track. That is, if $\tau'$ is a refinement resulting from splitting one cusp then we have
$$\ell^{\lambda}(\tau') =  \ell^{\lambda}(\tau) + \ell^{\lambda}(\hat{e}).$$  
See Figure \ref{pic splitting cusp} for some examples of splitting cusps. \\

\begin{figure}[h]
\includegraphics[width=1\textwidth]{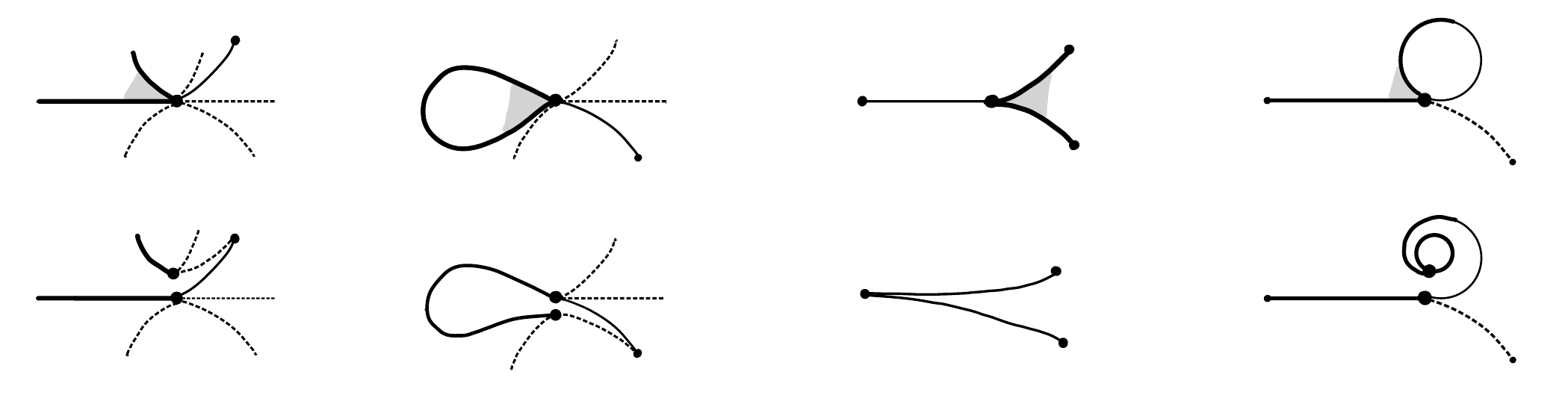}
\caption{Examples of the process of splitting a cusp. The top row are local pictures of a train track with a cusp shaded, the corresponding two half-edges in bold, and the half-edges which are not a common child dashed. The bottom row is the result of splitting the cusp.}
\label{pic splitting cusp}
\end{figure}
\medskip

\noindent Using this tool we define a few processes resulting in refinements. Recall that we say an edge $e$ is {\em $\lambda$-loopy} if it has a single vertex and $(e^+, e^-)$ is a $\lambda$-legal turn, where $e^+, e^-$ are the two half-edges of $e$. If $e$ is an edge with only one vertex and such that one half-edge is incoming and the other outgoing at that vertex but which is not $\lambda$-loopy (i.e. no leaf of $\lambda$ takes that turn) we call it a {\em fake loopy} edge. A {\em non-loopy} edge is one that is neither loopy nor fake loopy (that is, a non-loopy edge is one which is not an embedded circle).
\medskip

{\bf Combing a half-edge.} Let $e_0$ be a half-edge which is part of a non-loopy edge, and let $v_0$ be the vertex adjacent to $e_0$. We {\em comb} $e_0$ through the following steps: 
\begin{enumerate}
\item Split any adjacent cusps and set $e_0'=(\Phi^{\hat\tau,\tau})^{-1}(e_0)$ where $\hat\tau$ is the resulting simplicial refinement, and let $v_0'\in (\Phi^{\hat\tau,\tau})^{-1}(v_0)$ be the vertex adjacent to $e_0'$
\item Split all cusps at  $v_0'$.  
\item Remove any bivalent vertices.
\end{enumerate}

Letting $\tau'$ be the resulting refinement of $\tau$, note that $\Phi(e)=\Phi^{\tau',\tau}(e)$ is a concatenation of edges of $\tau$ for every edge $e$ of $\tau'$. In particular we have that any edge $e$ of $\tau'$ such that $\Phi(e)$ intersects the interior of $e_0$ satisfies 
$$\ell^{\lambda}(e) = \ell^{\lambda}(e_0)+\ell^{\lambda}(e')$$
for some child $e'$ of $e_0$. See Figure \ref{pic combing} for an example of combing an edge. 

\begin{figure}[h]
\includegraphics[width=1\textwidth]{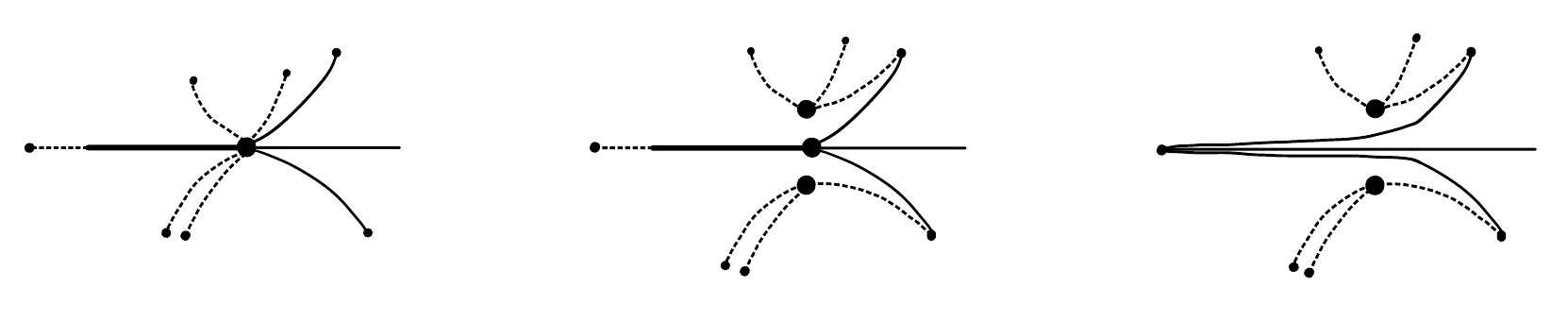}
\caption{Examples of combing a half-edge. Left: local picture with the half-edge to be combed in bold and its children represented by the solid lines. Middle: result after splitting the adjacent cusps. Right: End result, after splitting the remaining cusps and removing bivalent vertices.}
\label{pic combing}
\end{figure}

\medskip 

{\bf Unmasking a fake loopy edge.} Suppose $e_0$ is a fake loopy edge with vertex $v_0$ in $\tau$. There are two cases: either there is a $\lambda$-legal turn $\{e_1, e_2\}$ which crosses $e_0$ or there is not. In the second case, we simply replace $v_0$ with two vertices $v_0'$, $v_0''$ and replace $e_0$ with an edge connecting $v_0'$ and $v_0''$ such that the resulting train track still carries $\lambda$; see top row in Figure \ref{pic winding0}. In the second case, we do the same but connect $v_0'$ and $v_0''$ with an additional edge $e$ carrying segments that took the $\{e_1, e_2\}$ turn and then we comb $e$ (and finally we delete any bivalent vertices). The bottom row of Figure \ref{pic winding0} for an example. Note that the length $\ell^{\lambda}(\tau')$ of resulting refinement $\tau'$ either is the same as that of $\tau$ or 
$$\ell^{\lambda}(\tau') = \ell^{\lambda}(\tau)+\ell^{\lambda}(e_1)+\ell^{\lambda}(e_2)$$
where $e_1, e_2$ are two edges adjacent to $v_0$. 

\bigskip

Finally we consider the actual loopy edges. If $e$ is a $\lambda$-loopy edge and $e^+, e^-$ are its two half-edges, for every leaf $l$ of $\lambda$ for which $\Phi^{\lambda, \tau}(l)$ traverses $e$, $\Phi^{\lambda, \tau}(l)$ must take the turn $\{e^-, e^+\}$ a number of times. We define the {\em $\lambda$-winding number} of $e$ to be $w=k-1$ where $k$ is the largest integer such that there is a leaf of $\lambda$ which takes the turn $\{e^-,e^+\}$ $k$ times. Note that if $e$ has $\lambda$--winding number $w$, then every leaf $l$ of $\lambda$ such that $\Phi^{\lambda, \tau}(l)$ transverses $e$ takes the turn $\{e^-,e^+\}$ either $w$ or $w+1$ times. 

\begin{figure}[h]
\includegraphics[width=.8\textwidth]{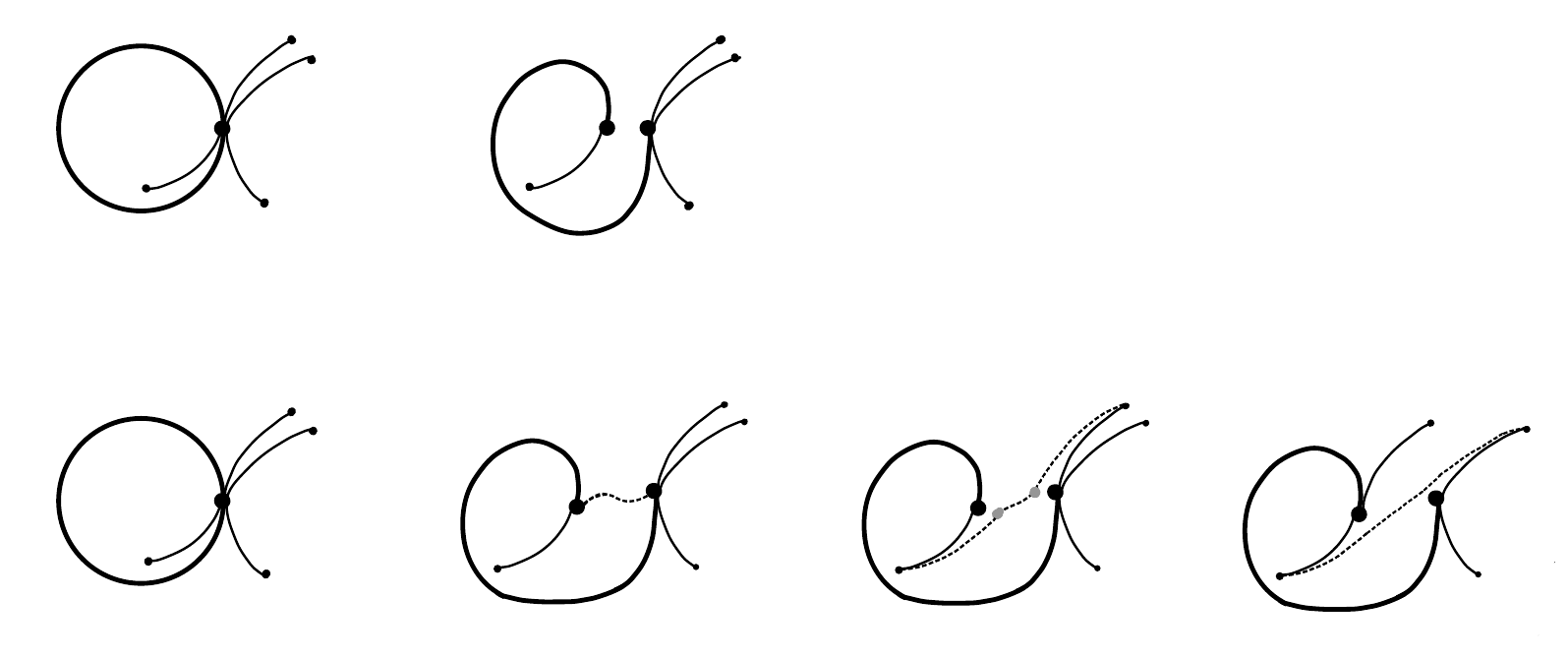}
\caption{The two possibilities of unmasking a fake loopy edge.}
\label{pic winding0}
\end{figure}

If the winding number $w$ is not too large, it will be helpful below to replace a loopy edge $e_0$ with a concatenation of $w$ edges  by {\em unlooping} it according to the following process: \\


{\bf Unlooping a loopy edge.} Suppose $e_0$ is a loopy edge with vertex $v_0$. We will construct a refinement $\tau'$ of $\tau$ such that $\Phi^{-1}(e_0)$ consists of (one or two) non-loopy edges.  Choose an incoming half-edge $e_0^+$ of $e_0$, and denote the other half-edge by $e_0^-$. There is an edge $e$ incoming at $v_0$ and adjacent to $e_0^+$ such that $(e, e^-_0)$ is a $\lambda$-legal turn. Split the cusp $v_0$ adjacent to $e_0^+$ and $e$. This results in a simplicial refinement $\tau'$ of $\tau$ with
$$\ell^{\lambda,\tau'}(e_0)=\ell^{\lambda,\tau}(e_0)-\ell(e_0).$$
Note that at this point we have not added any $\lambda$-length to the train track. Now, delete any bivalent vertices and continue inductively until we arrive at a refinement where the edge corresponding to $e_0$ is a fake loopy edge. We complete the process by unmasking it and deleting any bivalent vertices. The length of the refined track track is related to the length of $\tau$ just like in the case of unmasking fake loopy edges. 

Intuitively, the reader should picture the loopy edge $e$ having been unwrapped into a non-loopy edge of length approximately $\ell^{\tau}(e)$ and generically also a second non-loopy edge obtained through the unmasking in the last step.  See Figure \ref{pic windingk}.


\begin{figure}[h]
\includegraphics[width=.5\textwidth]{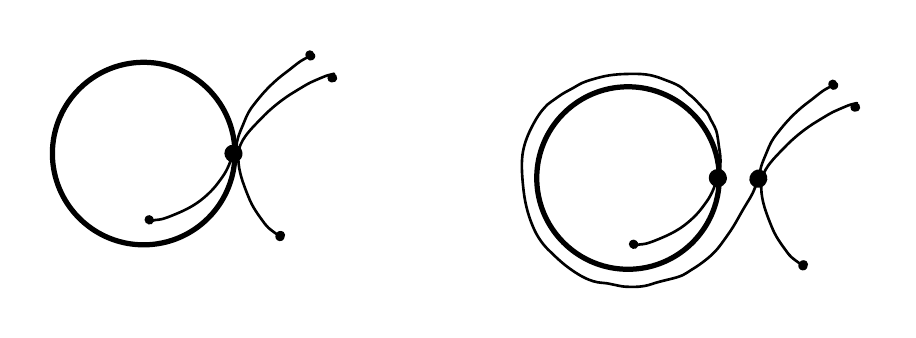}
\caption{The first step in unlooping a loopy edge.}
\label{pic windingk}
\end{figure}

\medskip

We will apply the above processes to define our main procedure. As outlined in the beginning of this appendix, the idea is to modify a given train track trading any short edges for longer ones while controlling the total length of the resulting train track.

\medskip

{\bf Main Procedure.} Let $\tau_0$ be any (almost geodesic) train track carrying $\lambda$ and recall that $\ell^{\lambda}(e)=\ell^{\lambda,\tau}(e)$, $\ell^{\lambda}(\tau)$ and $m^\lambda(\tau)$ were defined in \eqref{eq talked to Hugo1} and \eqref{eq talked to Hugo2}.

We start by adding bivalent vertices to any non-loopy edge $e$ with $\ell(e)>4m^{\lambda}(\tau_0)$ such that it gets subdivided into a number of edges each having length $2m^{\lambda}(\tau_0)$ except possibly one which has length at least $2m^{\lambda}(\tau_0)$ but less than $4m^{\lambda}(\tau_0)$. In the resulting train track unmask all fake loopy edges, and if this results in an edge of length greater than $4m^{\lambda}(\tau_0)$, then add bivalent vertices to subdivide it as above.  Finally, unloop any loopy edge $e$ with $\ell^{\lambda}(e)\leq 4m^{\lambda}(\tau_0)$. Denote the resulting refinement by $\hat{\tau}_0$. Note that $m^{\lambda}(\hat{\tau}_0) = m^{\lambda}(\tau_0)$, that $\hat{\tau}_0$ has the same number of edges of length less than $2m^{\lambda}(\tau_0)$ as $\tau_0$ does, and that for each (fake) loopy edge we have (unmasked or) unlooped we have increased the total length by at most $8m^{\lambda}(\tau_0)$ and hence  $\ell^{\lambda}(\hat{\tau}_0)\leq\ell^{\lambda}(\tau_0)+9\vert\chi(S)\vert\cdot8m^{\lambda}(\tau_0)$ where the factor $9\vert\chi(S)\vert$ is a (very non-optimal) bound on how many loopy edges $\tau_0$ can have.

Now let $e_0$ be an edge realizing $m^{\lambda}(\tau_0)$ in $\hat{\tau}_0$ and note that it is non-loopy and any adjacent edge has length at most $4m^{\lambda}(\tau_0)$ and at least $m^{\lambda}(\tau_0)$. Let $e^+_0$ be a half-edge of $e_0$ and $v_0$ its vertex. Combing $e^+_0$ (and then removing all remaining bivalent vertices) we obtain a refinement $\tau_0^1$ of $\tau_0$ carrying $\lambda$. Note that the number of edges $e$ in $\tau_0^1$ with $\ell^{\lambda}(e)=m^{\lambda}(\tau_0)$ is strictly less than the number of edges $e$ in $\tau_0$ with $\ell^{\lambda}(e)=m^{\lambda}(\tau_0)$, that the number of edges $e$ in $\tau_0^1$ with $\ell^{\lambda}(e)<2m^{\lambda}(\tau_0)$ is strictly less than the number of edges $e$ in $\tau_0$ with $\ell^{\lambda}(e)<2m^{\lambda}(\tau_0)$, and that $\ell^{\lambda}(\tau_0^1)\geq\ell^{\lambda}(\tau_0)$. Intuitively, the reader can think of any edge in $\tau_0^1$ as corresponding to a concatenation of edges in $\hat{\tau}_0$ and any edge whose image contains $e_0$ is the concatenation of at least two and hence has length at least $2m^{\lambda}(\tau_0)$. 
Now, since combing an edge means splitting cusps and there are at most $6\vert\chi(S)\vert$ cusps in $S\setminus \hat{\tau}_0$ we have $\ell^{\lambda}(\tau_0^1)\leq\ell^{\lambda}(\hat{\tau}_0)+6\vert\chi(S)\vert\cdot4m^{\lambda}(\tau_0)\leq \ell^{\lambda}(\tau_0)+96\vert\chi(S)\vert\cdot m^{\lambda}(\tau_0)$. 

Now, if $m^{\lambda}(\tau_0^1)<2m^{\lambda}(\tau_0)$ we repeat the above process, replacing $\tau_0$ with $\tau_0^1$. We continue inductively as long as $m^{\lambda}(\tau_0^n)<2m^{\lambda}(\tau_0)$ where $\tau_0^n$ denotes the refinement resulting from the $n^{th}$ step. Note that in each step the number of edges of length less than $2m^{\lambda}(\tau_0)$ decreases and hence since there are finitely many edges, the procedure ends in finitely many steps. That is, there is a $k\leq9\vert\chi(S)\vert$ such that $m^{\lambda}(\tau_0^k)\geq2m^{\lambda}(\tau_0)$; we set $\tau_0^k:=\tau_1$.  Note that we have $\ell^{\lambda}(\tau_0)<\ell^{\lambda}(\tau_1)< \ell^{\lambda}(\tau_0)+1000\vert\chi(S)\vert\cdot m^{\lambda}(\tau_0)$.\\ 
\medskip

\noindent We are finally ready to prove the proposition: 

\begin{proof}[Proof of Proposition \ref{prop-useful traintrack}]
Let $\lambda$ be a geodesic lamination without closed leaves and $\tau_0$ any train track carrying it. Let $L\geq\ell^{\lambda}(\tau_0)$. As already mentioned after the statement of the proposition, up to replacing $\tau_0$ by a refinement, we can assume that $\tau_0$ is an almost geodesic train track and $\ell^{\lambda}(\tau_0)\geq L$. 

Apply the Main Procedure to $\tau_0$ resulting in the refinement $\tau_1$ where $m^{\lambda}(\tau_1)\geq2m^{\lambda}(\tau_0)$ and with $\ell^{\lambda}(\tau_1)< \ell^{\lambda}(\tau_0) + 1000\vert\chi(S)\vert\cdot m^{\lambda}(\tau_0)$. Continuing inductively, we obtain the refinement $t_k$ at the $k^{th}$ step satisfying 
$$m^{\lambda}(\tau_k)\geq 2^km^{\lambda}(\tau_0)$$
and 
\begin{align*}
L<\ell^{\lambda}(\tau_k)&< \ell^{\lambda}(\tau_0) + 1000\vert\chi(S)\vert\cdot\sum_{i=0}^{k-1}m^{\lambda}(\tau_i)\\
&< \ell^{\lambda}(\tau_0) + 1000\vert\chi(S)\vert\cdot m^{\lambda}(\tau_k)\cdot\sum_{i=1}^{k}2^{-i}\\
&< \ell^{\lambda}(\tau_0) + 1000\vert\chi(S)\vert\cdot m^{\lambda}(\tau_k).
\end{align*} 
It follows that for any $C>1000\vert\chi(S)\vert$ there exists a $k$ such that $\tau_k$ is $C$-uniform. 

It remains to show that there is a {\em generic} such train track. Note that by choosing $k$ large enough above, we can assume $\tau_k$ is uniform and all its edges having enormous $\lambda$-lengths since $m^{\lambda}(\tau_k)$ grows exponentially with $k$. The idea is then to add more vertices along these very long edges and spread out the edges so that every vertex is trivalent at the non-loopy edges, and to partly unwrap any loopy edge so that its vertex is 4-valent. We sketch this process and leave the precise details to the reader; in particular, for simplicity we assume there are no loopy edges. Note that there are at most $E=18\cdot\vert\chi(S)\vert$ half-edges in any train track and hence any vertex can be at most $E$-valent. Choose $k$ large enough so that every edge in the $C$-uniform train track $\tau_k$ has length at least $2^E\cdot L$. Let $v$ be a vertex which has maximal valence in $\tau_k$, which we assume is greater than 3. Add a bivalent vertex at the midpoint of every edge incident at $v$, then split a (any) cusp at $v$, and finally delete any remaining bivalent vertex. The maximal valency of a vertex in the resulting refinement is either strictly less than in $\tau_k$ or it remains the same but has one less vertex of this valency. Moreover, every edge has length at least $2^{E-1}\cdot L$ and the total length has increased by at most $\ell^{\lambda}(\tau_k)/2$. We repeat the process until (in finitely many steps $n\leq E$) every vertex is trivalent. The resulting refinement $\tau$ satisfy $m^{\lambda}(\tau) \geq m^{\lambda}(\tau_k)/2^n$ and $\ell^{\lambda}(\tau)\leq \ell^{\lambda}(\tau_k) + \ell^{\lambda}(\tau_k)\cdot\sum_{i=1}^n 2^{-i}< 2\ell^{\lambda}(\tau_k)$. It follows that $\tau$ is generic and $C'$-uniform for $C'=2^{E+1}\cdot C$.  

\end{proof}
  
\end{appendix}

\bibliographystyle{plain}
\bibliography{ref-nonorientable}

\end{document}